\documentclass[11pt,reqno]{amsart}
\usepackage{amsmath,amsfonts,amssymb,mathrsfs}
\usepackage{amssymb,mathrsfs,graphicx,amsmath,mathtools}
\usepackage{enumitem}
\usepackage{ifthen}
\usepackage[colorlinks=true,citecolor=blue]{hyperref}
\usepackage{xcolor}
\usepackage{graphicx}
\usepackage[margin=1.3in]{geometry}
\usepackage{caption}
\usepackage{sidecap}
\usepackage{rotating}
\usepackage{color}
\usepackage{enumerate}
\DeclareMathOperator*{\esssup}{ess\,sup}
\usepackage{cleveref}
\usepackage{float}

\provideboolean{shownotes} 
\setboolean{shownotes}{true} 
\newcommand{\margnote}[1]{
\ifthenelse{\boolean{shownotes}}%
{\marginpar{\raggedright\tiny\texttt{#1}}}%
{}%
}
\newcommand{\hole}[1]{
\ifthenelse{\boolean{shownotes}}%
{\begin{center} \fbox{ \rule {.25cm}{0cm} \rule[-.1cm]{0cm}{.4cm}
\parbox{.85\textwidth}{\begin{center} \texttt{#1}\end{center}} \rule
{.25cm}{0cm}}\end{center}} {} }

 \numberwithin{equation}{section}
\newtheorem{theorem}{Theorem}[section]
\newtheorem{definition}{Definition}[section]
\newtheorem{corollary}{Corollary}[section]

\newtheorem{lemma}{Lemma}[section]
\newtheorem{proposition}{Proposition}[section]

\theoremstyle{definition}
\newtheorem{remark}{Remark}[section]

\newcommand{\R}{\mathbb R}
\newcommand{\N}{\mathbb N}

\newcommand{\bq}{\begin{equation}}
\newcommand{\eq}{\end{equation}}

\newcommand{\pa}{\partial}

\newcommand{\dt}{\partial_t}

\newcommand{\nn}{\notag}

\newcommand{\F}{\mathcal{F}}

\newcommand{\D}{\mathcal{D}}

\newcommand{\sigmam}{\sigma_{\textnormal{max}}}

\DeclareMathOperator*{\essinf}{ess\,inf}
\DeclareMathOperator*{\essosc}{ess\,osc}

\usepackage{stackengine,scalerel}
\newcommand\gele{\mathrel{\hstretch{2}{%
  \stackanchor[1pt]{\scriptscriptstyle\geq}{\scriptscriptstyle\leq}}}}
\stackMath

\newcommand\lege{\mathrel{\hstretch{2}{%
  \stackanchor[1pt]{\scriptscriptstyle\leq}{\scriptscriptstyle\geq}}}}

\title[Regularity of Aggregation-Confinement-Diffusion with Saturation]{ Aggregation-Confinement-Diffusion Evolutions with Saturation: Regularity and Long-Time Asymptotics}

\author[Y. Alamri]{Yousef Alamri}
\address{Applied Mathematics and Computational Sciences (AMCS), Computer, Electrical and Mathematical Sciences and Engineering Division (CEMSE), King Abdullah University of Science and Technology (KAUST), Thuwal, 23955-6900, Kingdom of Saudi Arabia}{} 
\email{yousef.alamri@kaust.edu.sa}

\begin{document}

\allowdisplaybreaks

\subjclass[2020]{Primary 35B65, 35B40, 35K65. Secondary 35B45}

\keywords{Aggregation-Diffusion, Degenerate Mobility, Gradient Flow, Regularity.}

\begin{abstract}
We establish H\"older regularity for the weak solution to a degenerate diffusion equation in the presence of a local (drift) potential and nonlocal (interaction) term, posed in a bounded domain with no-flux boundary conditions. The degeneracy is due to saturation, i.e., it occurs when the solution reaches its maximal value. As a byproduct of the established regularity and the underlying dissipative structure of the evolution, we prove the uniform convergence of contractive solutions to a stationary state as $t\to \infty$.

\end{abstract}

\date{\today}

\maketitle



\section{Introduction and main results}
We study the regularity and long-time asymptotic of some density $\rho=\rho(x,t)$ governed by the equation 
\begin{equation} \label{PDE_gen}
    \partial_t \rho = \nabla \cdot \left( \rho \sigma(\rho) \nabla \dfrac{\delta \F}{\delta \rho} \right) \ \ \ \text{in} \ \Omega \times [0,\infty),
\end{equation}
with a nonnegative initial datum $\rho(x,0) = \rho_0$ bounded above by some maximal density $\rho_{\text{max}}$, and no-flux boundary conditions. The spatial domain $\Omega$ is some open subset of $\R^N$, $N \geq 1$, with a piecewise smooth boundary, the \emph{saturation} $s\mapsto \sigma(s)$ is a given continuous function vanishing at $\rho_{\text{max}}$, and $\delta \F/\delta \rho$ is the variational derivative of the free energy functional
\begin{equation} \label{F_func}
    \F[\rho] = \int_\Omega U(\rho) + \rho V + \dfrac{1}{2} \rho (W * \rho) dx.
\end{equation}
Here, $U: \R \to \R$ denotes the energy density that dictates the nature of the diffusion in the model (e.g., linear vs. nonlinear), while $V, W: \Omega \to \R$ are sufficiently smooth functions representing a confining potential and an interaction kernel, respectively. The convolution term $W*\rho$ is understood in a bounded domain sense. See \Cref{assumps} below for precise assumptions.

In the absence of saturation (i.e., $\sigma(\rho) \equiv 1$), the pair \eqref{PDE_gen}-\eqref{F_func} represents standard aggregation-diffusion models, with an additional external force/drift term due to the local potential $V$, that find applications in biology, physics, and population dynamics, see \cite{carrillo2019aggregation,bailo2024aggregation} for recent and comprehensive reviews of such models and their motivations. Without saturation, the evolution \eqref{PDE_gen} defines a gradient flow of the functional \eqref{F_func} on the space of probability measures on $\Omega$ with respect to the $L^2$-Wasserstein metric \cite{otto2001geometry}.  Particular instances of such models include the Fokker--Planck equation \cite{jordan1998variational,carrillo2001entropy} and McKean--Vlasov equation \cite{chayes2010mckean, carrillo2020long}, and equations involving porous medium-type diffusion \cite{chayes2013aggregation,carrillo2021phase}.

Physically faithful models require the incorporation of nonlinearity in the quantity of interest; indeed, the inclusion of the saturation function, which renders the mobility term $\rho \sigma(\rho)$ in \eqref{PDE_gen} nonlinear as opposed to the case without saturation, is an example. It is called for to take into account the volume-filling effect in chemotaxis models \cite{painter2002volume,burger2006keller}, prevention of overcrowding in population models \cite{hillen2001global, Zamponi2017}, or phase segregation in physics \cite{giacomin1997phase}, to name a few. While such nonlinearity is not manifested in the associated energy functional \eqref{F_func}, it enters into its dissipation rate over time, thus playing a crucial role in the steady state profile of the solution. Furthermore, the saturation function imposes natural global bounds on the solution, determined by the bounds of the initial datum. While a generalization of the $L^2$-Wasserstein metric has been proposed to accommodate the nonlinear mobility in \eqref{PDE_gen} \cite{dolbeault2009new, carrillo2010nonlinear,lisini2010class}, the lack of geodesic convexity with respect to this proposed metric in the presence of nontrivial potential and interaction term limits its applicability.

A series of works has studied the well-posedness of the solution to the one-dimensional version of \eqref{PDE_gen}-\eqref{F_func} via approximation by particle schemes. For instance, the nonlocal transport equation with saturation ($U \equiv V \equiv 0$) has been studied in \cite{di2019deterministic}, the local-nonlocal transport equation with generic nonlinear mobility in \cite{fagioli2022gradient}, and the aggregation-diffusion problem with one degeneracy and saturation in \cite{fagioli2018solutions}. More recently, in higher dimensions, the existence of a $C_0$-semigroup of $L^1$-contractions of certain weak solutions constructed by approximation, well-posedness of the minimizers of \eqref{F_func}, and long-time behaviors for solutions to \eqref{PDE_gen} with a sufficiently smooth and convex $U$ and no interaction term were established in \cite{carrillo2024aggregation}. See also the recent preprint \cite{gomez2025convergence} where an existence theory for the case $W \neq 0$ was developed. Numerical methods and illustrations of transient and long-time profiles for this class of problems are presented in \cite{bailo2023boundpreserving,carrillo2024structure,carrillo2024aggregation}.

The objective of this work is to establish the uniform continuity in space and time of weak solutions to \eqref{PDE_gen}-\eqref{F_func}. As an application of such uniform regularity estimates, we demonstrate the uniform convergence of contractive weak solutions to a stationary state in the limit $t \to \infty$. Whereas our focus will be on diffusions of Boltzmann-type, that is, $U(\rho):= \rho (\log \rho - 1)$, the techniques developed herein extend to other classes of diffusion (e.g., porous medium); see \Cref{Two-degen} below.

\subsection{Assumptions} \label{assumps}
\noindent
Upon computing the variational derivative of \eqref{F_func}, the problem of interest reads 
\begin{equation} \label{PDE}
\begin{cases}
    \partial_t \rho - \nabla \cdot \big( \sigma(\rho) \nabla \rho + \rho \sigma(\rho) \nabla (V +  W * \rho) \big) =0     & \text{in} \ \Omega_T,   \\
    \big( \sigma(\rho) \nabla \rho + \rho \sigma(\rho) \nabla (V +  W * \rho) \big) \cdot {\mathbf{n}} = 0 & \text{on} \ S_T, \\
    \rho(x,0) = \rho_0(x)& \text{in} \ \ \Omega,  
\end{cases}
\end{equation}
where $\Omega_T = \Omega \times (0,T]$ and $S_T = \pa \Omega \times [0,T]$, for some $0 < T < +\infty$. Here $\bf{n}$ denotes the outward unit normal on $\pa \Omega$. 
We denote the parabolic boundary of $\Omega_T$ by $\partial_p\Omega:= S_T \cup (\Omega \times \{0\})$.
The convolution term is understood in the bounded domain sense: $\nabla W * \rho(x) = \int_{\Omega} \nabla W(x-y) \rho(y) dy$. As stated earlier, we assume that the initial datum $\rho_0$ is nonnegative and bounded above by some maximal positive density $\rho_{\text{max}}$. We further assume the following:

\begin{itemize} [label={},labelsep=-0.03cm,leftmargin=*]   
     \item \textbf{ S1)}  $\sigma: [0,\rho_{\text{max}}] \to [0,\infty)$ is a continuous function such that $\sigma(s) >0$ when
     
     \ \ \ \ $s \in [0,\rho_{\text{max}})$ and  $\sigma(\rho_{\text{max}}) = 0$ with the fixed finite quantity
         \begin{align*}
         \sigma_{\rm{max}} := \max_{s \in \left[0, \rho_{\text{max}} \right]} \sigma(s). 
         \end{align*}
     \item  \textbf{ S2)} There exist constants $0 < c_0 \leq c_1$ such that
     \begin{equation*}
         c_0 \Theta(\rho_{\text{max}}-s) \leq  \sigma(s) \leq c_1 \Theta(\rho_{\text{max}}-s),
     \end{equation*}
     
     \ \ \ \ where $\Theta(s) = s^\beta$ for some constant $\beta > 0$.
      \item \textbf{ V)} \ The confinement potential $V \in  W^{1,\infty}(\Omega)$.
    \item \textbf{ W)} The interaction kernel $W \in  W^{1,\infty}(\Omega)$.
\end{itemize}
The constant $\sigmam$ depends on the choice of the saturation function. A typical choice for the saturation function is that of a monotone decreasing type, such as  
\begin{equation} \label{saturation-power}
    \sigma(\rho) = (\rho_{\text{max}} - \rho)^m 
\end{equation} 
for $m > 0$; nonetheless, the range of the constants $c_0$ and $c_1$ in $\bf{S2}$ allows for accommodating non-monotone saturations as long as they behave like a decreasing power near the saturation value, $\rho_{ \text{max}}$. The prototypical example of the confining and interaction potentials is $V(x) = W(x) = |x|^2/2$. The well-posedness and regularity of solutions to \eqref{PDE} with local and nonlocal potentials belonging to a lower integrability class is an interesting problem and is left for future investigation.

\subsection{Statement of results}

\subsubsection{H\"older Regularity}

Our main result is the interior H\"older regularity for weak solutions to \eqref{PDE}. The proof relies on the method of intrinsic scaling \cite{dibenedetto1993degenerate,urbano1930method}. Such a technique has been used to establish H\"older regularity for the porous medium equation \cite{di1979regularity}, parabolic $p$-Laplace-type equations \cite{dibenedetto1993degenerate,dibenedetto2012harnack}, and doubly nonlinear parabolic equations \cite{porzio1993holder}. We adapt the technique to accommodate a generic saturation function as well as local potential and nonlocal interaction terms.

For a compact subset $K \subset \Omega_T$, introduce the parabolic distance from $K$ to $\partial_p \Omega$:
\begin{equation*}
    {\rm{dist}} (K; \partial_p \Omega) = \inf_{\substack{(x,t) \in K \\ (y,s) \in \partial_p \Omega }} \left \{ |x-y| + |t-s|^{1/2} \right \}.
\end{equation*}
 
\begin{theorem} \label{inter-Holder}
    Let $\rho$ be a bounded weak solution for \eqref{PDE} in $\Omega_T$, then $\rho$ is locally H\"older continuous. More precisely, there exist constants $\Gamma>1$ and $\alpha \in (0,1)$ such that for every compact set $K \subset \Omega_T$, 
        \begin{equation*}
        \left|\rho(x_1,t_1) - \rho(x_2,t_2) \right| \leq \Gamma \left( 
        \dfrac{|x_1 - x_2| + |t_1 - t_2|^{1/2}}{\mathrm{dist} (K;\partial_p \Omega)}\right)^\alpha
    \end{equation*}
    for every pair $(x_1,t_1),(x_2,t_2) \in  K$. The constants $\alpha$ and $\Gamma$ are independent of $(x_1,t_1),(x_2,t_2)$ and can be determined a priori only in terms of the data of the problem.
\end{theorem}
    The theorem is a consequence of the oscillation decay established in \Cref{oscil-decay} and a standard covering argument \cite[Theorem 4.10]{urbano1930method}.
    \begin{remark} \label{bndry}
         As \eqref{PDE} has no-flux boundary condition, if $\partial \Omega$ is of class $C^{1,\lambda}$ for some $\lambda \in (0,1)$, standard adaptations and localizations of the energy estimates developed in the present work near the boundary enable extending the regularity up to $\overline{\Omega} \times (0, T]$; see \cite[Sections 1-(iii) and 13 of Ch. III]{dibenedetto1993degenerate}.
    \end{remark}

    
\begin{remark} \label{Two-degen}
    H\"older regularity of the solution still holds if the energy density  $U(s)= s (\log s-1)$ is replaced by a porous medium-type density $U(s)=s^{m}/(m-1)$ for $m>1$ as long as the saturation function is of a power type, i.e., \eqref{saturation-power}. In the case of the latter, the diffusion coefficient is the continuous function $s \mapsto ms^{m-1} \sigma(s)$ 
    that is degenerate at two values: $s=0$ and $s= \rho_{
    \mathrm{max}}$. H\"older continuity in such a regime without the confinement and interaction terms has been established in \cite{urbano2001holder} with energy estimates that accommodate the two degeneracies. Nonetheless, the task of assimilating the potential and interaction terms can be performed as presented in this work. 
\end{remark}

\subsubsection{Convergence to steady-state}
For evolutions with an energy-dissipative structure, the stabilization to equilibrum as $t \to \infty$ is of interest. In the recent work \cite{carrillo2024aggregation}, provided $W=0$, the authors have demonstrated the global-in-time existence of a particular weak solution for the pair \eqref{PDE_gen}-\eqref{F_func} that satisfies energy dissipation and $L^1$-contraction properties.
As an application of the demonstrated regularity herein, we prove the convergence of such solutions to a limit that satisfies the stationary problem in the uniform topology. In particular, it was shown in \cite[Section 2.4] {carrillo2024aggregation} that if $V \in C^2(\overline{\Omega})$ and $W=0$, then the solution converges in $L^1(\Omega)$ to a unique limiting state as $t \to \infty$. Here, we obtain the following time-asymptotic result:
\begin{theorem}  \label{conv-thm}
     Let $\rho$ be an $L^1$-contractive weak solution of \eqref{PDE} in the sense of \cite{carrillo2024aggregation} with $V \in C^2(\overline{\Omega})$, $W=0$, and an initial datum satisfying $0 \leq \rho_0(x) \leq \rho_{\mathrm{max}}$, a.e. in $\Omega$. Then, \begin{equation*}
        \lim_{t \to \infty} \| \rho(\cdot,t) - \rho_\infty \|_{L^\infty(\Omega)} = 0,
    \end{equation*}
    where $\rho_\infty$ is a weak stationary solution of \eqref{PDE}.
\end{theorem}

The proof of \Cref{conv-thm} is given in \Cref{stationary-sec}.

\section{H\"older regularity} \label{regularity-sec}

We prove interior H\"older regularity for the weak solution to $\eqref{PDE}$. The challenge in this setting is that the principal part of the equation vanishes when the density $\rho$ reaches its maximal value $\rho_{\text{max}}$.  For simplicity in this section, we will consider the normalized solution  $\tilde{\rho} = \rho/\rho_\text{max}$ of \eqref{PDE} corresponding to the normalized initial datum $\tilde{\rho}_0 = \rho_0/ \rho_\text{max}$ such that, without loss of generality, we are restricted to the case $\rho_\text{max} = 1$. We suppress the tilde throughout.

A key idea in the regularity proof is to work with an $N$-dimensional space-time cylinder vertexed at some $(x_0,t_0) \in \Omega_T \subset \R^{N+1}$ whose dimensions are dynamically rescaled to reflect the degeneracy in the equation. The remaining reasoning follows an idea originally due to DeGiorgi \cite{de1957sulla} in the context of degenerate elliptic equations and later generalized to deal with degenerate parabolic equations by DiBenedetto \cite{di1986local}, culminating in the so-called \emph{intrinsic scaling} method \cite{dibenedetto1993degenerate,urbano1930method}: if the set where the equation degenerates can be confined within a small portion of the parabolic cylinder (cf. \eqref{first-alt}), then this set can be somehow controlled and the equation is thus essentially nondegenerate. On the other hand, if the set of degeneracy occupies a large portion of the parabolic cylinder (cf. \eqref{second-alt}), then such a set can be made comparable to its radius.

\subsection{Notion of weak solution}

We begin by defining the notion of a weak solution, which we shall consider for the regularity proof.
\begin{definition}\label{weak_sol_def} A weak solution to $\eqref{PDE}$ is a measurable function $\rho= \rho(x,t)$ such that
\begin{align*}
    &\rho \in C(0,T; L^2(\Omega)) \cap L^\infty(\Omega_T),\\
    &\Sigma(\rho): = \int^\rho_0 \sigma(s)ds \in L^2(0,T;H^1(\Omega)),
\end{align*}
 and for any $\varphi = \varphi(x,t) \in H^1(\Omega_T)$ satisfying $\varphi(\cdot,T)=0$,
  \begin{equation} \label{weak_form}
    \int_{\Omega}  \rho_0 \varphi(x,0) dx    + \iint_{\Omega_T} \rho \pa_t \varphi dx dt -\iint_{\Omega_T}   \Big\{  \sigma(\rho) \nabla \rho + \rho  \sigma(\rho) \nabla \left(V  + W * \rho \right) \Big\} \cdot \nabla \varphi dx dt = 0.
\end{equation}
\end{definition}

Energy estimates for the weak solutions are fundamental ingredients for the intrinsic scaling approach. To construct such estimates, test functions that depend on the solution itself are needed. A difficulty arises as weak solutions possess insufficient time regularity. This can be overcome by working with a time-averaged (alternatively, -mollified via convolution) version of the solution. Concretely, given a function $v \in L^1(\Omega_T)$ and $0 < h < T$, we consider the Steklov average of $v$ 
\begin{equation*}  \label{stek}
v_h(\cdot,t) := \begin{cases}  \dfrac{1}{h} \displaystyle \int_t^{t+h} v(\cdot, s)  ds,& 0 < t < T-h,\\ 
 0, & t > T - h,
 \end{cases}
\end{equation*}
for $0<t<T$ whose time derivative is  
\begin{equation*}
    \dt v_h = \dfrac{v(x,t+h) - v(x,t)}{h}.  
\end{equation*}

 Consult \cite{dibenedetto1993degenerate,wu2001nonlinear} for standard properties of the Steklov average. Of special use herein, we will employ  
 \cite[Lemma 2.1.2]{wu2001nonlinear} to show that $v_h \to v$ as $h \to 0$ in an appropriate topology. Putting this to use, we can introduce the following equivalent and more convenient definition of the weak solution given in \Cref{weak_sol_def}.
\begin{definition}
    For every compact subset $\Tilde{\Omega} \subset \Omega$ and $t \in (0,T-h]$, a weak solution of $\eqref{PDE}$ is a function $\rho_h$ satisfying,
\begin{equation}\label{weak_form_stek}
    \int_{\Tilde{\Omega}  \times \{t\}} \Big \{ \pa_t \rho_h \phi + (\sigma(\rho) \nabla \rho)_h \cdot \nabla \phi +  (\rho \sigma(\rho) )_h  \nabla (V +  W * \rho_h) \cdot \nabla \phi  \Big \}dx = 0.
    \end{equation}
     for any $\phi \in H^1_0(\Tilde{\Omega})$.
\end{definition}

\subsection{Intrinsic geometry} \label{intrin-geom}
For $x_0 \in \R^N$, define the $N$-dimensional ball centered at $x_0$ with radius $r>0$ as 
\begin{equation} \label{ball}
    B_r(x_0) := \left\{x \in \R^N:  |x-x_{0}| < r \right\},
\end{equation}
 and the parabolic cylinder with a vertex at $(x_0,t_0) \in \R^{N+1}$, radius $r$, and height $t >0$ as 
\begin{equation} \label{cyl}
    (x_0,t_0) + Q(t, r) := B_r(x_0) \times (t_0-t,t_0).
\end{equation}
with its Lebesgue measure  
\begin{equation}\label{cyl-measure}
|(x_0,t_0) + Q(t,r)| = t |B_r(x_0)|, 
\end{equation}
where $|B_r(x_0)| = C(N) \cdot r^N$.
Up to a translation, we may assume that $(x_0,t_0)$ coincides
with the origin and work with the notations $B_r$ and $Q(t,r)$ in place of \eqref{ball} and \eqref{cyl}, respectively. 
\noindent
For $0 < R < 1$ and sufficiently small $\epsilon>0$ such that $Q(R^{2-\epsilon},R) \subset \Omega_T$, we define the following quantities:
$$\mu_+ := \esssup_{Q(R^{2-\epsilon},R)} \rho, \ \ \ \mu_- := \essinf_{Q(R^{2-\epsilon},R)} \rho, \ \ \ \omega := \essosc_{Q(R^{2-\epsilon},R)} \rho = \mu_+ - \mu_-.$$

We also define the \emph{intrinsically scaled} parabolic cylinder
\begin{equation} \label{scl}
    Q(\theta_1 R^2,R) \subset \Omega_T \ \ \ \ \text{with} \ \ \    \dfrac{1}{\theta_1}   := c_0 \Theta\left(\dfrac{\omega}{4}\right),
\end{equation}
where $c_0$ and $\Theta(\cdot)$ are specified by assumption $\mathbf{S2}$. It is noteworthy to observe that the time length of the cylinder is scaled by the inverse of the lower bound on the saturation function; thus, the degeneracy is absorbed in the dimensions of the cylinder. Furthermore, the rescaled cylinder reduces to the standard parabolic cylinder associated with the heat equation if no degeneracy is present in the equation, i.e., if we fix $\beta=0$ and $c_0=c_1=1$ in assumption $\bf{S2}$, see \Cref{fig:Cyl} for a schematic.
\begin{figure}[htp]
\includegraphics[width=10cm]{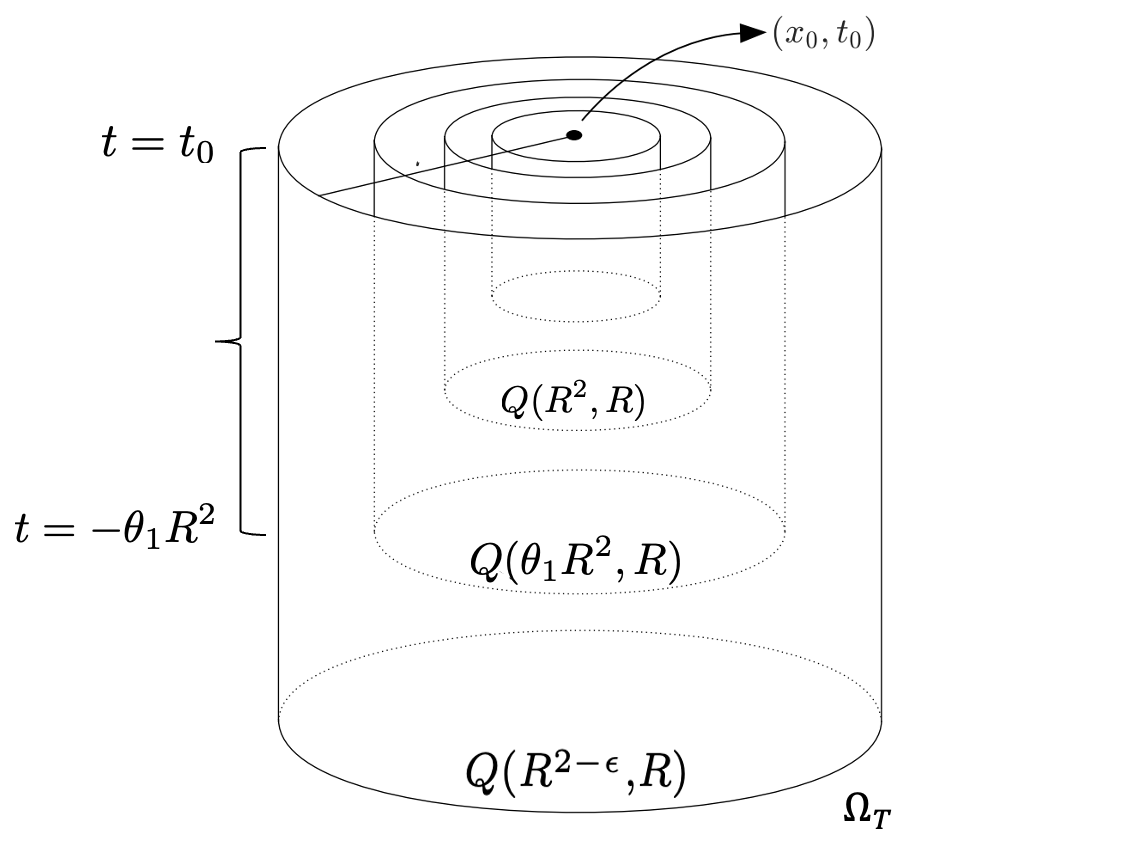}
     \caption{ The standard parabolic cylinder $Q(R^2,R)$ contained in the rescaled cylinder $Q(\theta_1 R^2,R)$.}
    \label{fig:Cyl}
\end{figure}

As $0 \leq \rho \leq 1$ by assumption, we may fix $\mu^- := 0$ and $\mu^+ :=1$ throughout for convenience. We enforce that
$$Q(\theta_1 R^2,R) \subset Q(R^{2-\epsilon},R) \subset \Omega_T,$$
entailing
\begin{equation} \label{assump1}
    R^{\epsilon} \leq  \dfrac{1}{\theta_1} \leq c_0 ,
\end{equation}
and consequently
\begin{equation} \label{assump2}
   \essosc_{ Q(\theta_1 R^2,R)} \rho\leq \omega
\end{equation}
which is a crucial assumption for the iteration process leading to the oscillation decay. 

\begin{remark}  \label{WLOG}
If \eqref{assump1} does not hold, then it would hold that $\theta_1^{-1} < R^\epsilon$, which implies that the oscillation $\omega$ tends to zero as the radius $R$ does. The desired oscillation decay in such a situation is obvious.
\end{remark}


\begin{remark} \label{LOT}
    We will frequently encounter quantities of the form $C_i R^{\hat{b}_iN\kappa}\omega^{-b_i}$ for a fixed $\kappa >1$ and some constants $C_i,b_j>1$, $\hat{b}_j>0$ that can be determined a priori from the data of the problem. We may enforce, without loss of generality, that 
    \begin{equation*}
        C_i R^{\hat{b}_iN\kappa}\omega^{-b_i} \leq 1.
    \end{equation*}
    If not, we have that $\omega < CR^{\epsilon_0}$, where $C:= \left( \max_{i} C_i\right)^{1/\min_{j}b_j}$ and $\epsilon_0 :=  \min_j \hat{b}_j N \kappa/\max_j b_j  $, in which case the oscillation decay is obvious.
\end{remark}
\begin{remark}
    Future calculations reveal the explicit values of the constants in the previous remark: $\kappa = 1+4/N$, $\min_j \hat{b}_j = 2/(N+4)$,  $\min_j b_j = \beta+1$,  $\max_j b_j = 2(\beta+1)$, and $\max_i C_i = \max\{c_0^{-2},1\} 2^{4\beta+2s_*}$ where $s_* \geq 4$ is some fixed finite number that depends on the data.  
\end{remark}

\subsection{Notations and auxiliary results} \label{conventions} Given $S \subset \mathbb{R}^N$ measurable, we write $|S|$ to denote the Lebesgue
measure of $S$ in $\mathbb{R}^N$. We denote by $C$ a generic positive constant or $C(X)$ to emphasize the dependence of $C$ on $X$.  

\noindent
For a function $v: \Omega \to \mathbb{R}$ and two real numbers $k_1 < k_0$, we will adopt the following notation
\begin{equation*}
    [k_1 < v < k_0] := \{ x \in \Omega : k_1 < v(x) < k_0 \},
\end{equation*}
and the usual conventions 
$$v_+ = \max\{v,0\}, \ \ \ \ \ v_- := \max\{-v,0\}.$$
Given constants $a,b,c$, with $a> c> 0$, we will employ the following nonnegative logarithmic function:
\begin{align} \label{logfunc}
    \psi^\pm_{\{a,b,c\}}(s) &:= \left( \log \dfrac{a}{(a+c) -(s-b)_{\pm}} \right)_+ \nn \\
    &= \begin{cases}
        \log \dfrac{a}{(a+c) \pm (b-s)} & \text{if } b \pm c \lessgtr s \lessgtr b \pm (a+c), \\
        0 & \text{if } s \lege b \pm c,  
    \end{cases}
\end{align}
whose first derivative is 
\begin{align}
    \left(\psi^\pm_{\{a,b,c\}}\right)'(s) &= \begin{cases}
         \dfrac{1}{(b-s) \pm (a+c)} & \text{if } b \pm c \lessgtr s \lessgtr b \pm (a+c), \\
        0 & \text{if } s \lessgtr b \pm c,  \end{cases}
\end{align}
with $(\psi^\pm)' \gele 0$ for which the following (in)equalities hold \cite{dibenedetto1993degenerate}:
\begin{subequations} \label{logfuncEstim}
    \begin{align}
    \psi^\pm &\leq \log \dfrac{a}{c}, \\
    (\psi^\pm)' &\leq  \dfrac{1}{c}, \\
    (\psi^\pm)'' &= [(\psi^\pm)']^2.
    \end{align}
\end{subequations}
For $k \in \mathbb{R}$, $t, r>0$, and a function $\rho=\rho(x,t)$, we will work with truncated versions of $\rho$, that is $(\rho-k)_\pm$, satisfying
\begin{equation} \label{assump3}
    H_{\rho,k}^\pm := \esssup_{Q(t,r)} \ (\rho-k)_\pm \leq \delta_0,
\end{equation}
for some determined $\delta_0 > 0$; and the sets
\begin{equation}\label{A+-}
            A_{k,r}^\pm(s) :=\{ x \in B_{r}(x_0): \pm \rho(x,s) > \pm k   \}.
\end{equation}
We recall the parabolic Sobolev space 
\begin{equation*}
    V^2(\Omega_T) := L^\infty(0,T; L^2(\Omega)) \cap L^2(0,T; H^{1}(\Omega)),
\end{equation*}
endowed with the norm 
\begin{equation} \label{Vnorm}
    \|v\|^2_{V^2(\Omega_T)} := \esssup_{0 < t < T} \|v\|^2_{L^2(\Omega)} + \|\nabla v\|^2_{L^2(\Omega_T)}.
\end{equation}
In addition to the constant $\sigmam$ given in assumption $\bf{S1}$, we fix the following frequent quantity
\begin{equation} \label{BigLam}
    \Lambda := 2\cdot \max\{\|\nabla V\|_{L^\infty(\Omega)},\|\nabla W\|_{L^\infty(\Omega)} \},
\end{equation}
which is finite by assumptions $\bf{V}$ and $\bf{W}$.
We employ the following convergence result quoted from \cite[Sec. 4.2  of Ch. I]{dibenedetto1993degenerate}.

\begin{lemma} [Fast geometric convergence] \label{fast-conv-lemma}
    Let $\{Y_n\}$ and $\{Z_n \}$, for $n = 0,1,2,\dots$, be sequences of positive numbers satisfying the recursive inequalities 
    \begin{equation*}
        \begin{cases}
            Y_{n+1} \leq Cb^n \left(Y_{n}^{1+\upsilon} + Y_n^\upsilon Z_{n}^{1+ \kappa}  \right) , \\
            Z_{n+1} \leq Cb^n \left(Y_{n} + Z_{n}^{1+ \kappa} \right),
        \end{cases}
    \end{equation*}
    where $C,b>1$ and $\kappa, \upsilon>0$ are given numbers. If
    \begin{equation*}
        Y_0 + Z_0^{1+\kappa} \leq (2C)^{-\tfrac{1+\kappa}{d}} b^{-\tfrac{1+\kappa}{d^2}}, \ \ \ \ \text{where} \ \ \ d := \min\{\kappa,\upsilon\},
    \end{equation*}
    then $\{Y_n\}$ and $\{Z_n \}$ tend to zero as $n \to \infty$.
\end{lemma}
Finally, we recall DeGiorgi's lemma (cf. \cite [Lemma 2.2 of Ch. I]{dibenedetto1993degenerate}).

\begin{lemma}  [DeGiorgi's isoperimetric inequality] \label{DegLemma}Let $v \in W^{1,1}(B_r(x_0))$, where $B_r(x_0) \subset \Omega \subset \R^N$. Let $k_1<k_0$ be two real numbers. Then, there exists a constant $C=C(N)>0$ such that
    \begin{equation*} 
   (k_0-k_1)|v>k_0| \leq C(N) \dfrac{r^{N+1}}{|v<k_1|} \int_{[k_1 < \rho < k_0]} |\nabla v| dx.
\end{equation*}
\end{lemma}



\subsection{Energy estimates}
We shall consider two complementary cases. Given $\nu \in (0,1)$, to be determined depending only on the data, the weak solution $\rho$ of \eqref{PDE} satisfies either 
\begin{equation}
\label{first-alt}
     \left|\left\{(x,t) \in Q(\theta_1 R^2,R): \rho(x,t) > 1-\dfrac{\omega}{2}\right\}\right|   \leq \nu |Q(\theta_1 R^2,R)|,
\end{equation}
or  
\begin{equation} \label{second-alt}
     \left|\left\{(x,t) \in Q(\theta_1 R^2,R): \rho(x,t) < \dfrac{\omega}{2}\right\}\right|   < (1-\nu)|Q(\theta_1 R^2,R)|.
\end{equation}

The aim is to show that in either case, decreasing the size of the parabolic cylinder of $Q(\theta_1 R^2, R)$ is accompanied by a reduction in the oscillation of the solution. We stress that such a reduction is controlled by a factor that does \emph{not} depend on the oscillation $\omega$ nor the radius $R$ of the parabolic cylinder. 

As mentioned earlier, the intrinsic scaling approach relies on crafting particular energy estimates employed to reach desired conclusions when either \eqref{first-alt} or \eqref{second-alt} holds. In the following proposition, we construct a Caccioppoli-type estimate for the weak solution.

 \begin{proposition} \label{energy-estim-prop}
    Let $0 \leq \rho \leq 1$ be a weak solution for \eqref{PDE} in $\Omega_T$. There exists constants $C(\cdot) >0$, depending only on the data, such that for every parabolic cylinder $Q(\tau,r) \subset \Omega_T$, every $k \in \mathbb{R}$, and every nonnegative, piecewise, smooth cutoff function $\zeta$ in $Q(\tau,r)$, there holds
\begin{align}\label{energy-estim}
            \esssup_{-\tau < t < 0} \int_{B_r \times \{t\}} &(\rho - k)_\pm^2 \zeta^2 dx +  \iint\limits_{Q(\tau,r)} 
         \sigma(\rho) |\nabla (\rho - k)_\pm|^2 dx dt \nn \\
         &\leq \int_{B_r \times \{-\tau\}} (\rho - k)_\pm^2 \zeta^2 dx + C \iint\limits_{Q(\tau,r)} \sigma(\rho)  (\rho - k)_\pm^2 |\nabla \zeta|^2 dx dt \nn \\
         &+ C \iint\limits_{Q(\tau,r)}  (\rho - k)_\pm^2 \zeta \partial_t \zeta + C(\sigmam,\Lambda)  \left[  \int_{-\tau }^{0} |A_{k,r}^\pm(t)|dt \right].
        \end{align}
    \end{proposition}
    
\begin{proof}
    In the weak formulation \eqref{weak_form_stek}, we take the test function $\phi = \pm (\rho_h - k)_\pm \zeta^2$, and integrate in time over the range $(-\tau,t_\circ)$ for $t_\circ \in (-\tau,0)$:  
    \begin{align*} 
    &\underbrace{\iint\limits_{Q(\tau,r)} \dt \rho_h \left[\pm \left(\rho_h - k \right)_\pm \zeta^2 \right] dx dt}_{\mathcal{I}_1} + \underbrace{\iint\limits_{Q(\tau,r)}  (\sigma(\rho) \nabla \rho)_h \cdot \nabla \left[\pm \left(\rho_h - k \right)_\pm \zeta^2 \right] dx dt}_{\mathcal{I}_2} \nn \\
        &+ \underbrace{ \iint\limits_{Q(\tau,r)} (\rho \sigma(\rho))_h \nabla \big[ V +  W*\rho_h \big] \cdot \nabla \left[\pm \left(\rho_h - k \right)_\pm \zeta^2 \right] dx dt}_{\mathcal{I}_3} = 0.
    \end{align*}
    We analyze each term separately. Integrating by parts and passing to the limit as $h \to 0$ (see \cite[Lemma 2.1.2]{wu2001nonlinear}), we get
    \begin{align} \label{local-energy-I}
        \mathcal{I}_1 &=\dfrac{1}{2}  \iint\limits_{Q(\tau,r)} \partial_t \left[(\rho_h - k)_\pm^2 \right] \zeta^2   dx dt \nn \\
        &\to   \dfrac{1}{2} \int_{B_r \times \{t_\circ\}} (\rho - k)_\pm^2 \zeta^2 dx - \dfrac{1}{2} \int_{B_r \times \{-\tau \}} (\rho - k)_\pm^2 \zeta^2 dx \nn \\
        &- \iint\limits_{Q(\tau,r)} (\rho - k)_\pm^2 \zeta \partial_t \zeta dx dt.
    \end{align}
As to the second and third terms, we again pass to the limit $h \to 0$ and estimate from below:
\begin{align*}   
    \mathcal{I}_2 &\to \pm  \iint\limits_{Q(\tau,r)} \sigma(\rho) \nabla \rho \cdot \nabla (\rho-k)_\pm \zeta^2 dx dt  \pm  2 \iint\limits_{Q(\tau,r)} \sigma(\rho) \nabla  \rho (\rho-k)_\pm \zeta \cdot \nabla \zeta dx dt \nn \\
    &\geq  (1-\varepsilon_1) \iint\limits_{Q(\tau,r)} \sigma(\rho) |\nabla (\rho-k)_\pm|^2 \zeta^2  dx dt  - \dfrac{1}{\varepsilon_1} \iint\limits_{Q(\tau,r)} \sigma(\rho) (\rho-k)^2_\pm |\nabla \zeta|^2 dx dt,
\end{align*}
where we have applied Young's inequality to the second term in $\mathcal{I}_2$: 
\begin{align*} 
    \left| \pm 2 \iint\limits_{Q(\tau,r)} \sigma(\rho) \nabla (\rho-k)_\pm (\rho-k)_\pm \zeta \cdot \nabla \zeta dx dt \right| 
    &\leq \varepsilon_1  \iint\limits_{Q(\tau,r)} \sigma(\rho) |\nabla (\rho-k)_\pm|^2 \zeta^2  dx dt \\
    &+ \dfrac{1}{\varepsilon_1}  \iint\limits_{Q(\tau,r)} \sigma(\rho) (\rho-k)^2_\pm |\nabla \zeta|^2 dx dt.
\end{align*}
Likewise, recalling the quantity $\sigmam$ fixed in assumption $\bf{S1}$, $\Lambda$ given in \eqref{BigLam}, and the assumption $0 \leq \rho \leq 1$, we estimate
\begin{align*} 
    |\mathcal{I}_3| &\leq  \iint\limits_{Q(\tau,r)} \Lambda \sigma(\rho) |\nabla (\rho-k)_\pm| \zeta^2 dx dt + 2  \iint\limits_{Q(\tau,r)} \Lambda \sigma(\rho)  (\rho-k)_\pm  \zeta |\nabla \zeta| dx dt  \nn \\
     & \leq \dfrac{1}{2 \varepsilon_2} \iint\limits_{Q(\tau,r)} \sigma(\rho) |\nabla (\rho-k)_\pm|^2  \zeta^2 dx dt + \dfrac{1}{\varepsilon_3} \iint\limits_{Q(\tau,r)} \sigma(\rho) (\rho-k)^2_\pm |\nabla \zeta|^2  dx dt \nn \\
     &+ \left( \dfrac{\varepsilon_2}{2} + \varepsilon_3 \right) \sigmam \Lambda \iint\limits_{Q(\tau,r)} \chi_{\{\pm \rho> \pm k\}} dx dt.
\end{align*}
Choosing  $\varepsilon_1 = 1/4,  \varepsilon_2 = 2$, $\varepsilon_3 = 1$ and recalling the definition \eqref{A+-} yield
\begin{align} \label{I23}
    \mathcal{I}_2 + \mathcal{I}_3 &\geq \dfrac{1}{2} \iint\limits_{Q(\tau,r)} \sigma(\rho) |\nabla (\rho-k)_\pm|^2 \zeta^2  dx dt -5 \iint\limits_{Q(\tau,r)} \sigma(\rho) (\rho-k)^2_+ |\nabla \zeta|^2 dx dt \nn \\
    &-  2 \sigmam \cdot \Lambda  \left[  \int_{-\tau }^{0} |A_{k,r}^\pm(t)|dt \right].
    \end{align}
 Rearranging \eqref{local-energy-I}-\eqref{I23} and noting that $t_\circ \in (-\tau,0)$ is arbitrary yield the desired estimate.
\end{proof}

The next proposition demonstrates a logarithmic energy estimate for the weak solution. Such an estimate will be employed to expand, in time, certain claims established in parabolic subcylinders to hold in the entirety of the containing cylinder. The logarithmic function $\psi^\pm$ and its properties have been listed in \Cref{conventions}.

 \begin{proposition} \label{log-energy-estim-prop}
    Let $0 \leq \rho \leq 1$ be a weak solution for \eqref{PDE} in $\Omega_T$  and $H_{\rho,k}^\pm$ be as defined in \eqref{assump3}. There exist constants $c, C(\cdot) >0$, depending only on the data, such that for every parabolic cylinder $Q(\tau,r) \subset \Omega_T$, every $k \in \mathbb{R}$, and every nonnegative, piecewise smooth cutoff function $\zeta$ in $Q(\tau,r)$, there holds
 \begin{align}\label{log-energy-estim}
     \esssup_{-\tau < t < 0} &\int_{B_r \times \{t\}}  [\psi^\pm(\rho)]^2 \zeta^2 dx \leq  \int_{B_r \times \{-\tau\}}  [\psi^\pm(\rho)]^2 \zeta^2 dx  \nn \\
     &+ C \iint\limits_{Q(\tau,r)} 
         \sigma(\rho) \psi^\pm(\rho) |\nabla \zeta|^2  dx dt  + \dfrac{C(\sigmam,\Lambda)}{c^2} \left( 1+ \log \dfrac{H^\pm_{\rho,k}}{c} \right) \left[  \int_{-\tau }^{0} |A_{k,r}^\pm(t)|dt \right].
        \end{align}
\end{proposition}
\begin{proof}
    Recall the logarithmic function \eqref{logfunc} and set 
    \begin{equation*}
        \psi^\pm  := \psi^\pm_{\{H_{\rho,k}^\pm,(\rho-k)_\pm,c\}}(\rho)
    \end{equation*}
    for some $0< c < H_{\rho,k}^\pm$. In \eqref{weak_form_stek}, take the test function 
    \begin{equation*}
        \phi = \partial_{\rho_h} [\psi^\pm(\rho_h)]^2 \zeta^2 = 2 \psi^\pm (\psi^\pm)' \zeta^2, 
    \end{equation*}
    to obtain
    \begin{align*} 
    &\underbrace{\iint\limits_{Q(\tau,r)} \dt \rho_h \left[ 2 \psi^\pm (\psi^\pm)' \zeta^2 \right] dx dt}_{\mathcal{I}_1} + \underbrace{\iint\limits_{Q(\tau,r)}  (\sigma(\rho) \nabla \rho)_h \cdot \nabla \left[ 2 \psi^\pm (\psi^\pm)' \zeta^2 \right] dx dt}_{\mathcal{I}_2} \nn \\
        &+ \underbrace{ \iint\limits_{Q(\tau,r)} (\rho \sigma(\rho))_h \nabla \big[ V +  W*\rho_h \big] \cdot \nabla \left[ 2 \psi^\pm (\psi^\pm)' \zeta^2 \right] dx dt}_{\mathcal{I}_3} = 0.
    \end{align*}
    Sending $h \to 0$, we obtain, for any $t_\circ \in (-\tau, 0)$,
    \begin{align} \label{I1-log}
        \mathcal{I}_1 \to \int_{B_r \times \{t_\circ\}} [\psi^\pm(\rho)]^2 \zeta^2 dx - \int_{B_r  \times \{-\tau\}} [\psi^\pm(\rho)]^2 \zeta^2 dx.
    \end{align}
    Considering the second term, take $h \to 0$ and apply Young's inequality to obtain
    \begin{align*} 
        \mathcal{I}_2 &\to 2 \iint\limits_{Q(\tau,r)}  \sigma(\rho) \nabla \rho \cdot \nabla \rho ( 1 + \psi^\pm)[(\psi^\pm)']^2 \zeta^2 dx dt + 4 \iint\limits_{Q(\tau,r)}  \sigma(\rho) \nabla \rho \cdot \nabla \zeta    \psi^\pm(\psi^\pm)' \zeta   dx dt \nn \\
        &\geq  2\iint\limits_{Q(\tau,r)}  \sigma(\rho) |\nabla \rho|^2( 1 + \psi^\pm)[(\psi^\pm)']^2 \zeta^2 dx dt \nn \\
        &-2  \iint\limits_{Q(\tau,r)} \varepsilon_1 \sigma(\rho) |\nabla \zeta|^2 dx dt - 2 \iint\limits_{Q(\tau,r)} \dfrac{1}{\varepsilon_1} \sigma(\rho) |\nabla \rho|^2 [\psi^\pm]^2[(\psi^\pm)']^2 \zeta^2 dx dt .
    \end{align*}
 The third term is estimated as  
\begin{align*}
    |\mathcal{I}_3| & \leq 2  \iint\limits_{Q(\tau,r)}  \Lambda \sigma(\rho) |\nabla \rho|   ( 1 + \psi^\pm)[(\psi^\pm)']^2 \zeta^2 dx dt  + 4 \iint\limits_{Q(\tau,r)}  \Lambda \sigma(\rho) 
    |\nabla \zeta| \psi^\pm(\psi^\pm)'  \zeta dx dt \nn \\
    & \leq \iint\limits_{Q(\tau,r)} \sigma(\rho) \left( \varepsilon_2 |\nabla \rho|^2 + \dfrac{1}{\varepsilon_2} \Lambda^2  \right)  ( 1 + \psi^\pm)[(\psi^\pm)']^2 \zeta^2 dx dt  \\
    &+ 2\iint\limits_{Q(\tau,r)}  \sigma(\rho) \left( \varepsilon_3 |\nabla \zeta|^2 + \dfrac{1}{\varepsilon_3} \Lambda^2 [\psi^\pm (\psi^\pm)']^2 \zeta^2 \right)   dx dt.
\end{align*}
Choosing $\varepsilon_1 = 2 \psi^\pm$, $\varepsilon_2 = 1$, $\varepsilon_3 = \psi^\pm$ and noting $\psi^\pm \leq (1+ \psi^\pm)$, we obtain the lower bound
\begin{align} \label{I23-log}
    \mathcal{I}_2 + \mathcal{I}_3 & \geq -6 \iint\limits_{Q(\tau,r)} \sigma(\rho)     \psi^\pm(\rho) |\nabla \zeta|^2  dx dt \nn \\
    &- 3\dfrac{ \sigmam \cdot \Lambda^2}{c^2} \left( 1+ \log \dfrac{H^\pm_{\rho,k}}{c} \right) \left[  \int_{-\tau }^{0} |A_{k,r}^\pm(t)|dt \right],
    \end{align}
    where we have used the estimates \eqref{logfuncEstim}. Collecting \eqref{I1-log}-\eqref{I23-log} yield the desired estimate.
\end{proof}

\subsection{Behavior near the saturation value} \label{first-alt-sec}
We begin by assuming that \eqref{first-alt} holds in the parabolic cylinder $Q(\theta_1 R^2, R)$. The aim of the following proposition is to show that the value of the solution $\rho$ is below the saturation value in a smaller parabolic cylinder, along with determining $\nu$. The oscillation reduction then follows immediately.  
\begin{proposition} \label{prop1}
    There exists a number $\nu \in (0,1)$, depending only on the data, such that if \eqref{first-alt} holds, then, 
    \begin{equation}  \label{first-alt-conclusion-fa}
        \rho(x,t) < 1- \dfrac{\omega}{4} \  \ \ \text{ a.e. in }  \ \ (x,t) \in Q\bigg(\theta_1 \left(\frac{R}{2}\right)^2,\frac{R}{2} \bigg).
    \end{equation}
\end{proposition}
\begin{proof}
         We shall work with the quantity $\rho_{\omega}:= \min\{ \rho,  1-\frac{\omega}{4}\}$ and proceed to construct an energy estimate for a truncated version of it. Define, for $n = 0,1,2,\dots$, the two sequences 
     \begin{align*}
        R_n &= \dfrac{R}{2} + \dfrac{R}{2^{n+1}},  \\
        k_n & =1 - \dfrac{\omega}{4} - \dfrac{\omega}{2^{n+2}}. 
    \end{align*}
    Construct a family of nested and shrinking parabolic cylinders $Q_n^1 := Q(\theta_1 R^2_n,R_n)$ along with the corresponding cutoff functions $0 \leq \zeta_n \leq 1$ in $Q_n^1$ satisfying:
     \begin{equation} \label{cutoffs1}
         \begin{cases}
             \zeta_n \equiv 1 \ \ \ \ \  \text{in} \ Q_{n+1}^1, \\ 
            \zeta_n \equiv 0 \ \ \ \  \ \text{on} \  \partial_p Q_n^1, \\
              0 \leq \partial_t \zeta_n \leq \dfrac{2^{2(n+2)}}{\theta_1 R^2}, \ \ \  |\nabla \zeta_n| \leq \dfrac{2^{n+2}}{ R}, \ \ |\Delta \zeta_n| \leq \dfrac{2^{2(n+2)}}{ R^2},
         \end{cases}
     \end{equation}
where $\partial_p Q_n^1$ denotes the parabolic boundary of $Q^1_n$. In \eqref{weak_form_stek}, take 
    \begin{equation*}
        \phi := \left( (\rho_{\omega})_h - k_n \right)_+ \zeta_n^2,
    \end{equation*}
    and integrate over the time interval $(-\theta_1 R^2_n,t_\circ)$ for $t_\circ \in (-\theta_1 R^2_n,0)$:
        \begin{align} \label{term1-fa}
    &\underbrace{\iint\limits_{Q_n^1} \dt \rho_h \left[ \left( (\rho_{\omega})_h - k_n \right)_+ \zeta_n^2 \right] dx dt}_{\mathcal{I}_1} + \underbrace{\iint\limits_{Q_n^1}  (\sigma(\rho) \nabla \rho)_h \cdot \nabla \left[ \left( (\rho_{\omega})_h - k_n \right)_+ \zeta_n^2 \right] dx dt}_{\mathcal{I}_2} \nn \\
        &+\underbrace{ \iint\limits_{Q_n^1} (\rho \sigma(\rho))_h \nabla \big[ V +  W*\rho_h \big] \cdot \nabla \big[\left( (\rho_{\omega})_h - k_n \right)_+ \zeta_n^2 \big] dx dt}_{\mathcal{I}_3} = 0.
    \end{align}
    Observe that $(\rho_\omega - k_n)_+$ is nonzero on the set $\{k_n < \rho \}$. Specifically, if $\rho_\omega = \rho < 1- \frac{\omega}{4}$, 
    \begin{equation} \label{obs2}
        (\rho_{\omega} -k_n)_+ = (\rho -k_n)_+ \leq  \dfrac{\omega}{2^{n+2}} \leq \dfrac{\omega}{4}.
    \end{equation}
    On the other hand,  if $ \rho  \geq 1- \frac{\omega}{4} = \rho_\omega > k_n$, then
        \begin{equation} \label{obs1}
        (\rho_{\omega} -k_n)_+ = \dfrac{\omega}{2^{n+2}} \leq \dfrac{\omega}{4} .
    \end{equation}
    Moreover, notice that $|\nabla (\rho_\omega - k_n)_+| \neq 0$ on the set $\{ k_n  < \rho < 1 -\frac{\omega}{4} \}$. With these observations, we proceed to estimate each term in \eqref{term1-fa}. Firstly, we have 
    \begin{align*}
        \mathcal{I}_1 &= \iint\limits_{Q_n^1} \dt \rho_h \left[ \left( (\rho_{\omega})_h - k_n \right)_+ \zeta_n^2 \right] \chi_{\left\{(\rho_w)_h = \rho_h\right\}} dx dt \\
        &+ \iint\limits_{Q_n^1} \dt \rho_h \left[\left( (\rho_{\omega})_h - k_n \right)_+ \zeta_n^2 \right] \chi_{\left\{(\rho_w)_h = 1-\frac{\omega}{4}\right\}} dx dt \\
        &= \dfrac{1}{2}  \iint\limits_{Q_n^1} \dt \left[ \left( (\rho_{\omega})_h - k_n \right)_+^2\right] \zeta_n^2  dx dt +  \left( \dfrac{\omega}{2^{n+2}} \right) \iint\limits_{Q_n^1} \dt \left[ \left( \rho_h - \Big[1-\dfrac{\omega}{4} \Big] \right)_+ \right]\zeta_n^2 dx dt.
    \end{align*}
    Next, we integrate by parts with respect to time, send $h \to 0$, and use \eqref{cutoffs1} to obtain
    \begin{align}\label{estim-I}
        \mathcal{I}_1 &\to \dfrac{1}{2}\int_{B_n \times \{t_\circ\}} (\rho_{\omega} - k_n)_+^2 \zeta_n^2 dx - \iint\limits_{Q_n^1} (\rho_{\omega} - k_n)_+^2 \zeta_n \dt \zeta_n \nn dx dt \\
        &+\left(\dfrac{\omega}{2^{n+2}}\right) \int_{B_n \times \{t_\circ\}} \left( \rho -\Big[1-\dfrac{\omega}{4} \Big] \right)_+ \zeta_n^2 dx \nn\\
        &-2\left(   \dfrac{ \omega}{2^{n+2}} \right) \iint\limits_{Q_n^1}  \left( \rho - \Big[1-\dfrac{\omega}{4} \Big]  \right)_+\zeta_n \dt \zeta_n dx dt \nn \\
        &\geq \dfrac{1}{2}\int_{B_n \times \{t_\circ\}} (\rho_{\omega} - k_n)_+^2 \zeta_n^2 dx - 3 \left(\dfrac{\omega}{4}\right)^2 \dfrac{2^{2(n+2)}}{\theta_1 R^2} \iint\limits_{Q_n^1}  \chi_{\{ \rho_{\omega} > k_n\}}dx dt, 
    \end{align}
    where the last inequality follows from the nonnegativity of the third term.
    Next, we estimate $\mathcal{I}_2$ in \eqref{term1-fa}. Passing to the limit $h \to 0$ and noting where $\nabla\left( \rho_{\omega}  - k_n \right)_+$ is nonzero yield
    \begin{align} \label{II1}
        \mathcal{I}_2 &\to \iint\limits_{Q_n^1} \sigma(\rho) \nabla \rho \cdot \nabla \left[\left( \rho_{\omega}  - k_n \right)_+ \zeta_n^2 \right] dx dt \nn\\
        &= \iint\limits_{Q_n^1} \sigma(\rho)  \nabla \rho \cdot \nabla \left[\left( \rho_{\omega}  - k_n \right)_+ \zeta_n^2 \right] \chi_{\left\{\rho_w = \rho \right\}} dx dt \nn \\
        &+ \iint\limits_{Q_n^1} \sigma(\rho)  \nabla \rho \cdot \nabla \left[\left( \rho_{\omega}  - k_n \right)_+ \zeta_n^2\right] \chi_{\left\{\rho_w = 1-\frac{\omega}{4}\right\}} dx dt \nn\\
        &= \iint\limits_{Q_n^1} \sigma(\rho_\omega) | \nabla \left( \rho_{\omega}  - k_n \right)_+|^2 \zeta_n^2 dx dt \nn \\
        &+ 2 \iint\limits_{Q_n^1} \sigma(\rho_\omega) \nabla\left( \rho_{\omega}  - k_n \right)_+ \cdot \nabla \zeta_n \zeta_n \left( \rho_{\omega}  - k_n \right)_+ dx dt \nn \\
        &+ 2 \left(\dfrac{\omega}{2^{n+2}}\right) \iint\limits_{Q_n^1} \sigma(\rho) \nabla \rho \cdot \nabla \zeta_n \zeta_n  \chi_{\{ \rho \geq 1-\frac{\omega}{4}\}} dx dt.
    \end{align}
    Applying Young's inequality to the second term on the rightmost side of the above expression
    \begin{align*} 
        &\left|  2 \iint\limits_{Q_n^1} \sigma(\rho_\omega) \nabla\left( \rho_{\omega}  - k_n \right)_+ \cdot \nabla \zeta_n \zeta_n \left( \rho_{\omega}  - k_n \right)_+ dx dt \right| \nn \\
        &\leq \dfrac{1}{4} \iint\limits_{Q_n^1} \sigma(\rho_\omega) | \nabla \left( \rho_{\omega}  - k_n \right)_+|^2 \zeta_n^2 dx dt + 4 \iint\limits_{Q_n^1} \sigma(\rho_\omega) \left( \rho_{\omega}  - k_n \right)^2_+ |\nabla \zeta_n|^2 dx dt.
    \end{align*}
    We arrive at the lower bound
    \begin{align} \label{II1-new}
        \mathcal{I}_2 &\geq \dfrac{3}{4}\iint\limits_{Q_n^1} \sigma(\rho_\omega) | \nabla \left( \rho_{\omega}  - k_n \right)_+|^2 \zeta_n^2 dx dt \nn - 4 \iint\limits_{Q_n^1} \sigma(\rho_\omega) \left( \rho_{\omega}  - k_n \right)^2_+ |\nabla \zeta_n|^2 dx dt \nn \\
        &+ \underbrace{2 \left(\dfrac{\omega}{2^{n+2}}\right) \iint\limits_{Q_n^1} \sigma(\rho) \nabla \rho \cdot \nabla \zeta_n \zeta_n  \chi_{\{ \rho \geq 1-\frac{\omega}{4}\}} dx dt}_{\mathcal{I}'_2}.
    \end{align}
    Now, we make the following observations: on the set $\{1-\frac{\omega}{2} \leq k_n < \rho_{\omega} = \rho < 1- \frac{\omega}{4} \}$, we have by assumption $\mathbf{S2}$
    \begin{align*}
        \sigma(\rho) &\geq c_0 \Theta\left(1-\rho\right) \geq c_0 \Theta\left(1-\Big[ 1-\dfrac{\omega}{4} \Big]\right) = \dfrac{1}{\theta_1}, \\
        \sigma(\rho) &\leq c_1 \Theta\left(1-\rho\right) \leq c_1 \Theta\left(1-\Big[ 1-\dfrac{\omega}{2} \Big]\right) \leq 2^{\beta} \dfrac{c_1}{c_0} \dfrac{1}{\theta_1},
    \end{align*}
    which, along with \eqref{cutoffs1} and \eqref{obs2}, allow us to estimate the first two terms on the rightmost side of \eqref{II1-new} from below. As to the third term, we introduce
    \begin{equation*}
        \Sigma_+(\rho) := \left[ \int_{1-\frac{\omega}{4}}^\rho \sigma (s) ds \right]_+, 
    \end{equation*}
    and notice that on the set $\{1 \geq \rho \geq 1-\frac{\omega}{4}\}$, we have again by assumption $\mathbf{S2}$,
    \begin{align} \label{Big-sigma}
        \Sigma_+(\rho) &\leq c_1 \Theta\left( 1-\Big[ 1-\dfrac{\omega}{4} \Big] \right) \left(1-\Big[ 1-\dfrac{\omega}{4} \Big]\right) = \dfrac{c_1}{c_0} \dfrac{1}{\theta_1}  \left(\dfrac{\omega}{4}\right).
    \end{align}
    Utilizing this estimate, integration by parts, and use of \eqref{cutoffs1} and \eqref{obs1}, the third term on the rightmost side of \eqref{II1-new} is estimated as follows
    \begin{align*}  
    \left| \mathcal{I}'_2 \right| 
    &\leq \left|2 \left(\dfrac{\omega}{2^{n+2}}\right) \iint\limits_{Q_n^1} \nabla \Sigma_+(\rho) \cdot \nabla \zeta_n \zeta_n   dx dt \right| \nn \\
        &= \left| -2 \left(\dfrac{\omega}{2^{n+2}}\right) \iint\limits_{Q_n^1}  \Sigma_+(\rho)  \left( \zeta_n |\Delta \zeta_n| + |\nabla \zeta_n|^2 \right) dx dt \right| \nn \\
        & \leq 4 \dfrac{c_1} {c_0} \dfrac{1}{\theta_1}  \left(\dfrac{\omega}{4} \right)^2\bigg(\dfrac{2^{2(n+2)}}{R^2}\bigg) \iint\limits_{Q_n^1}  \chi_{\{ \rho_{\omega} > k_n\}}  dx dt .
    \end{align*}
Combining the preceding estimates, we arrive at
    \begin{align} \label{estim-II}
        \mathcal{I}_2 &\geq \dfrac{3}{4} \dfrac{1}{\theta_1} \iint\limits_{Q_n^1} | \nabla \left( \rho_{\omega}  - k_n \right )_+|^2 \zeta_n^2 dx dt - 4 (1+2^{\beta}) \dfrac{c_1}{c_0} \left(\dfrac{\omega}{4}\right)^2  \bigg(\dfrac{2^{2(n+2)}}{\theta_1 R^2}\bigg) \iint\limits_{Q_n^1}  \chi_{ \{ \rho_\omega > k_n \}} dx dt.
    \end{align}
        As to $\mathcal{I}_3$ in \eqref{term1-fa}, upon letting $h \to 0$, applying Young's inequality, we obtain
    \begin{align} \label{estim-III}
        |\mathcal{I}_3| 
        & \leq  \dfrac{\sigmam \Lambda}{2 \varepsilon_1} \iint\limits_{Q_n^1} | \nabla \left( \rho_{\omega}  - k_n \right)_+|^2 \zeta_n^2 dx dt +  \dfrac{\sigmam \Lambda}{ \varepsilon_2} \iint\limits_{Q_n^1} \left( \rho_{\omega}  - k_n \right)^2_+  |\nabla \zeta_n|^2 dx dt \nn \\
        &+ \sigmam \Lambda \left( \dfrac{\varepsilon_1}{2}   + \varepsilon_2  \right) \iint\limits_{Q_n^1} \chi_{\{ \rho_{\omega} > k_n\}}  dx dt \nn \\
        & \leq  \dfrac{1}{4} \dfrac{1}{\theta_1} \iint\limits_{Q_n^1} | \nabla \left( \rho_{\omega}  - k_n \right)_+|^2 \zeta_n^2 dx dt \nn \\
        &+ \dfrac{c_1}{c_0} \left(\dfrac{\omega}{4}\right)^2  \bigg(\dfrac{2^{2(n+2)}}{ \theta_1 R^2}\bigg) \iint\limits_{Q_n^1}  \chi_{\{ \rho_\omega > k_n\}} dx dt \nn \\
        &+ 2\left(\sigmam \Lambda \right)^2 \theta_1  \left[ \int_{-\theta_1 R^2_n }^{0} |A_{k_n,R_n}^+(t)|dt \right],
        \end{align}
        where we chose $\varepsilon_1 = 2 \theta_1 \sigmam \Lambda$ and $\varepsilon_2 = \theta_1 \sigmam \Lambda c_0/c_1$, used $c_0 \leq c_1$, and defined 
        $$A_{k_n,R_n}^+(t) := \{ x \in B_n: \rho_\omega(x,t) > k_n\},$$
        with the shorthand $B_n:= B_{R_n}$. Subsituting $\eqref{estim-I},\eqref{estim-II}$, and $\eqref{estim-III}$ in \eqref{term1-fa} we arrive at the estimate
        \begin{align}\label{full-estim}
            \esssup_{-\theta_1 R_n^2 < t < 0} & \int_{B_n \times \{t\}} (\rho_\omega - k_n)_+^2 \zeta_n^2 dx + \dfrac{1}{\theta_1} \iint\limits_{Q_n^1}  
         |\nabla (\rho_{\omega}  - k_n)_+|^2 \zeta_n^2 dx dt \nn \\
          &\leq C(\beta,c_0,c_1) \left(\dfrac{\omega}{4}\right)^2  \bigg(\dfrac{2^{2(n+2)}}{\theta_1 R^2}\bigg) \iint\limits_{Q_n^1}  \chi_{\{ \rho_\omega > k_n\}} dx dt \nn \\
          &+ C(\sigmam,\Lambda) \theta_1^2 \left[ \dfrac{1}{\theta_1} \int_{-\theta_1 R^2_n }^{0} |A_{k_n,R_n}^+(t)|dt \right].
        \end{align}
Now, we wish to write the above estimate in terms of the norm \eqref{Vnorm}. To this end, we perform a change of the time variable by introducing \begin{equation} \label{change-of-var}
    \tilde{t} = \dfrac{t}{\theta_1},
\end{equation}
such that we have
\begin{equation*}
\tilde{\rho}_\omega(\cdot,\tilde{t}) := \rho_{\omega}(\cdot,t = \theta_1 \tilde{t}), \ \ \text{and} \ \ \ \tilde{\zeta}_n(\cdot, \tilde{t}) := \zeta_n(\cdot,t = \theta_1 \tilde{t}).
\end{equation*}
Furthermore, owing to the definition of $\theta_1$ given in \eqref{scl}, we estimate
\begin{equation*}
    \theta_1^{2}
    = c_0^{-2} 2^{4(\beta+1)} \omega^{-2\left(\beta+1\right)} \left(\dfrac{\omega}{4}\right)^2
    \leq   R^{-N \kappa} \left(\dfrac{\omega}{4}\right)^2,
\end{equation*}
where the inequality is due to the assumption $c_0^{-2} 2^{4(\beta+1)} \omega^{-2\left(\beta+1\right)} R^{N \kappa} \leq 1$ (cf. \Cref{LOT}) for some $\kappa>1$ to be fixed momentarily. Consequently, the estimate \eqref{full-estim} reads 
        \begin{align}\label{full-estim-bar}
            \|( \tilde{\rho}_\omega - k_n)_+ \tilde{\zeta_n}\|_{V^2(Q_n)} &\leq C(\beta,c_0,c_1) \left(\dfrac{\omega}{4}\right)^2  \bigg(\dfrac{2^{2(n+2)}}{ R^2}\bigg)  |\tilde{A}^+_n| \nn \\
            &+ C(\sigmam,\Lambda) R^{-N \kappa} \left(\dfrac{\omega}{4}\right)^2  \left[\int_{-R^2_n }^{0} |\tilde{A}^+_n( \tilde{t} )|d\tilde{t} \right],
        \end{align}
        with $Q_n := Q(R^2_{n},R_{n})$ and 
        \begin{equation*}
            \tilde{A}_n^+(\tilde{t} ) := \{ x \in B_n : \tilde{\rho}_\omega(x,\tilde{t}) > k_n \},  \ \ \
            |\tilde{A}^+_n| = \int_{-R^2_n }^{0} |\tilde{A}^+_n( \tilde{t} )|d\tilde{t} .
        \end{equation*}
    Observing that $R_{n+1}<R_{n}$, employing a Sobolev embedding theorem \cite[Corollary 3.1 of Ch. I]{dibenedetto1993degenerate}, and using \eqref{full-estim-bar}, we get the following inequalities
\begin{align} \label{norm-estim}
    \|( \tilde{\rho}_\omega - k_n)_+\|^2_{L^2(Q_{n+1})} & \leq  \|( \tilde{\rho}_\omega - k_n)_+ \tilde{\zeta}_n\|^2_{L^2(Q_n)} \nn \\
    & \leq C(N) \|( \tilde{\rho}_\omega - k_n)_+ \tilde{\zeta_n}\|^2_{V^2(Q_n)} |\tilde{A}^+_n|^{\frac{2}{N+2}} \nn \\
    & \leq C(\beta,c_0,c_1,\sigmam, \Lambda, N)  \left[\left(\dfrac{\omega}{4}\right)^2  \bigg(\dfrac{2^{2(n+2)}}{ R^2}\bigg)  |\tilde{A}^+_n|^{1+\frac{2}{N+2}} \right. \nn \\
    & \ \ \ \ \   \ \ \ \ \ \ \ \ \ \ \ \ \ \ \ \ \  \ \  \ \ \ \ \ \   \ \ \ \  \left. +  R^{-N \kappa} \left(\dfrac{\omega}{4}\right)^2  |\tilde{A}^+_n|^{1+\frac{2}{N+2}} \right].
\end{align}
As $k_n < k_{n+1} $, the leftmost side of \eqref{norm-estim} is estimated below by
\begin{equation}\label{norm-est-bel}
    \|( \tilde{\rho}_\omega - k_n)_+\|^2_{L^2(Q_{n+1})}  \geq |k_{n} - k_{n+1}|^2 |\tilde{A}^+_{n+1}| \geq  \dfrac{1}{2^{2(n+2)}} \left( \dfrac{\omega}{4} \right)^2 |\tilde{A}^+_{n+1}|  .
\end{equation} 
 Dividing \eqref{norm-estim}-\eqref{norm-est-bel} by $|Q_{n+1}|$ and introducing the quantities
\begin{equation*}
    Y_n := \dfrac{|\tilde{A}^+_n|}{|Q_n|}, \ \ \ \ \ \ Z_{n} := \dfrac{|\tilde{A}^+_n|^{\frac{1}{1+\kappa}}}{|B_n|},
\end{equation*}
we obtain
\begin{align}\label{recur1}
    Y_{n+1}  &\leq C(\beta,c_0,c_1,\sigmam,\Lambda, N) 4^{2n} (Y_n^{1+\frac{2}{N+2}} +   Y_n^{\frac{2}{N+2}} Z_{n}^{1+\kappa} ),
\end{align}
where we have employed the following estimates 
\begin{align*}
\dfrac{R_n}{R_{n+1}} \leq 2,  \ \ \ \ \ \ \ \ 
\dfrac{|Q(R^2_{n},R_{n})|^{1+\frac{2}{N+2}}}{|Q(R^2_{n+1},R_{n+1})|} \leq C(N) R^2, \ \ \  \ \ \ \ \ 
\dfrac{|Q(R^2_{n},R_{n})|^{\frac{2}{N+2}}}{|Q(R^2_{n+1},R_{n+1})|} \leq C(N) R^{-N} .
\end{align*}
Likewise, noting that $|B_{n+1}| \geq C(N) 2^{-N} R^{N}$, we estimate
\begin{align*} 
    \dfrac{1}{2^{2(n+2)}} \left( \dfrac{\omega}{4} \right)^2 Z_{n+1} \leq |k_n-k_{n+1}|^2 Z_{n+1}  &\leq \dfrac{1}{|B_{n+1}|}  \|( \tilde{\rho}_\omega - k_n)_+ \|_{L^{1+\kappa}(Q_{n+1})} \\
    &\leq \dfrac{1}{|B_{n+1}|}  \|( \tilde{\rho}_\omega - k_n)_+ \tilde{\zeta_n}\|_{L^{1+\kappa}(Q_{n})} \\
    & \leq C(N) R^{-N}  \|( \tilde{\rho}_\omega - k_n)_+ \tilde{\zeta_n}\|_{V^2(Q_{n})},
\end{align*}
where the last inequality is due to an application of a Sobolev embedding (cf. \cite[Lemma 2.1.1]{wu2001nonlinear}) with $\kappa:= 1 + 4/N$. Substituting the estimate \eqref{full-estim-bar}, we derive
\begin{align} \label{recur2}
    Z_{n+1} &\leq C(N) 2^{2n} R^{-N} \|( \tilde{\rho}_\omega - k_n)_+ \tilde{\zeta_n}\|_{V^2(Q_{n})} \nn \\
    &\leq C(\beta,c_0,c_1,\sigmam,\Lambda, N) 4^{2n} \left(   Y_n +  Z^{1+\kappa}_{n}\right).
\end{align}
By \Cref{fast-conv-lemma}, if \eqref{recur1} and \eqref{recur2} hold, and if, in addition,
\begin{equation} \label{conv-cond-1}
    Y_0 + Z_0^{1+\kappa} \leq (2 C(\beta,c_0, c_1,\sigmam,\Lambda, N))^{-\tfrac{1+\kappa}{d}}4^{-\tfrac{2(1+\kappa)}{d^2}}=: \nu \in (0,1),
\end{equation}
with $d := \min\{\kappa,2/(N+2) \} = 2/(N+2)$ for any $N\geq 1$, then it follows that $Y_n,Z_n \to 0$ as $n \to \infty$.
But \eqref{conv-cond-1} is precisely our initial assumption \eqref{first-alt}, provided we fix $\nu$ as indicated in \eqref{conv-cond-1}.
Hence, as $n \to \infty$, 
\begin{equation*}
    R_n \searrow \dfrac{R}{2}, \ \ \ \ \text{and} \ \ \ k_n \nearrow 1-\dfrac{\omega}{4},
\end{equation*}
and $|\tilde{A}^+_n| \to 0$. Finally, upon undoing the change of the time variable \eqref{change-of-var}, we conclude
\begin{align*}
   \left|\left\{(x,\tilde{t}) \in Q\bigg(\left(\frac{R}{2}\right)^2, \frac{R}{2} \bigg): \tilde{\rho}_\omega(x,\tilde{t}) \geq 1-\frac{\omega}{4}\right\}\right| &= \\
   \left|\left\{(x,t) \in Q\bigg(\theta_1 \left(\frac{R}{2}\right)^2,\frac{R}{2} \bigg): \rho(x,t) \geq 1- \frac{\omega}{4}\right\}\right| &= 0,
\end{align*}
whence \eqref{first-alt-conclusion-fa} follows.
\end{proof}

As a result of the foregoing analysis, whenever \eqref{first-alt} holds, we deduce the following oscillation reduction.

\begin{corollary} \label{oscRed1} Assume \eqref{assump1} and \eqref{first-alt} hold. Then, there exists a constant $\gamma_1 \in (0,1)$, depending solely on the data, such that 
    \begin{equation*} 
   \essosc_{ Q\left(\theta_1 \left(\frac{R}{2}\right)^2, \frac{R}{2}\right)} \rho \leq  \gamma_1 \omega.
\end{equation*}
\begin{proof}
    \Cref{prop1} asserts that 
    \begin{equation*}
        \esssup_{ Q\left(\theta_1 \left(\frac{R}{2}\right)^2, \frac{R}{2}\right)} \rho \leq  1-\dfrac{\omega}{4}.
    \end{equation*}
    It follows that
    \begin{equation*}
         \essosc_{ Q\left(\theta_1 \left(\frac{R}{2}\right)^2, \frac{R}{2}\right)} \rho \leq 1 - \dfrac{\omega}{4} - \mu^- = \omega - \dfrac{\omega}{4} - 0 = \gamma_1 \omega,
    \end{equation*}
    where
    $$\gamma_1 := \dfrac{3}{4}.$$
\end{proof}
\end{corollary}

\subsection{Behavior away from the saturation value} \label{second-alt-sec}
If \eqref{first-alt} does not hold in $Q(\theta_1 R^2, R)$, then \eqref{second-alt} must hold with the exact $\nu$ determined in \eqref{conv-cond-1}. We begin the analysis of this case by demonstrating the following auxiliary lemma which states that, at some time level $\hat{t}$, the portion of the ball $B_R$ in which $\rho(x)$ is close to its infimum, that is 0, can be compared to the radius of the ball. 

\begin{lemma} \label{lemma1}
    Suppose \eqref{second-alt} holds; then there exists a time level 
    \begin{equation} \label{t0-def}
        \hat{t} \in \left[-\theta_1 R^2, -\dfrac{\nu}{2} \theta_1 R^2 \right],
    \end{equation}
    such that 
    \begin{equation}\label{t0-estim}
        \left| \left\{ x \in B_R : \rho(x,\hat{t} \, ) < \dfrac{\omega}{2} \right\} \right| < \left( \dfrac{1-\nu}{1-\frac{\nu}{2}}\right) |B_R|.
    \end{equation}
    \begin{proof}
        Suppose  that the conclusion does not hold, then we have 
        \begin{align*}
            \left| \left\{ x \in Q(\theta_1 R^2, R) : \rho(x,t) < \dfrac{\omega}{2} \right\} \right| &= \int_{-\theta_1 R^2}^0 \left| \left\{ x \in B_R : \rho(x,\hat{t} \, ) < \dfrac{\omega}{2} \right\} \right| dt \\
            & \geq \int_{-\theta_1 R^2}^{-\frac{\nu}{2}\theta_1 R^2} \left|\left\{ x \in B_R : \rho(x,\hat{t} \, ) < \dfrac{\omega}{2} \right\} \right| dt \\
            & \geq \left( \dfrac{1-\nu}{1-\frac{\nu}{2}}\right) |B_R| \left( \theta_1 R^2 - \dfrac{\nu}{2}\theta_1 R^2 \right) \\
            & = \left( 1- \nu\right) |B_R| \theta_1 R^2 \\
            &= \left( 1- \nu\right) |Q(\theta_1 R^2,R)|
        \end{align*}
        contradicting \eqref{second-alt}.
    \end{proof}
\end{lemma}
Next, we expand the claim above so that it holds in the upper portion of the parabolic cylinder $Q(\theta_1 R^2, R)$, that is, up to $t = 0$. We will make use of the logarithmic energy estimate obtained in \Cref{log-energy-estim-prop}.

\begin{proposition} \label{prop3-4}
    There exists $s_1 \in \mathbb{N}$, depending only on the data, such that for every $t \in (\hat{t},0)$ we have 
    \begin{equation} \label{conc-1}
        \left|  \left\{ x \in B_R : \rho(x,t) < \dfrac{\omega}{2^{s_1}} \right\} \right| \leq \Big(1 - \left(\frac{\nu}{2} \right)^2 \Big) |B_R|.
    \end{equation}
    \begin{proof}
        Consider the logarithmic energy estimate \eqref{log-energy-estim} over the parabolic cylinder $Q(\hat{t} , R)$ with the choices
        \begin{equation*}
            \hat{k}:= \dfrac{\omega}{2}, \ \ \ \ c:= \dfrac{\omega}{2^{s_0+1}},
        \end{equation*}
        for $s_0 \in \mathbb{N}$ to be chosen. In the parabolic cylinder $Q(\hat{t},R)$, there holds
        \begin{equation*} 
         \hat{k}-\rho \leq H^-_{\rho,\hat{k}}: = \esssup_{Q(\hat{t},R)} \ (\rho-\hat{k})_- \leq \dfrac{\omega}{2}.
        \end{equation*}
        If $H^-_{\rho,\hat{k}} \leq  \frac{\omega}{4}$, then the conclusion \eqref{conc-1} follows trivially with the choice $s_1 = 2$. Hence, we hereafter assume  $H^-_{\rho,\hat{k}} > \frac{\omega}{4}$. Recall the logarithmic function \eqref{logfunc}:
        \begin{equation*}
           \psi^-(\rho) :=  \begin{cases}
        \log \left( \dfrac{H^-_{\rho,\hat{k}}}{H^-_{\rho,\hat{k}} + \rho - \hat{k} + c } \right) & \text{if } \rho < \hat{k}-c, \\
        0 & \text{if } \rho \geq \hat{k}-c.  
    \end{cases}
        \end{equation*}
    Observe also that  
\begin{equation*}  
  \dfrac{H^-_{\rho,\hat{k}}}{H^-_{\rho,\hat{k}} + \rho - \hat{k} + c } \leq \dfrac{H^-_{\rho,\hat{k}}}{c} \leq \dfrac{\tfrac{\omega}{2}}{c} = 2^{s_0},
\end{equation*}
        and by the estimates \eqref{logfuncEstim}, we have
        \begin{equation}\label{psi-esim}
             \psi^-(\rho) \leq \log \dfrac{H^-_{\rho,\hat{k}}}{c} = s_0 \log 2.
        \end{equation} 
    Choose a piecewise, smooth cutoff function $x \mapsto \zeta(x)$  satisfying 
 \begin{equation} \label{cutoffs2}
         \begin{cases}
             0 \leq \zeta(x) \leq 1 & \text{in } B_R, \\
             \zeta \equiv 1 & \text{in} \ B_{(1-\lambda)R}, \ \ \ \lambda \in (0,1), \\ 
            |\nabla \zeta| \leq \dfrac{1}{\lambda R}.
         \end{cases}
\end{equation}
    Applying assumption $\bf{S2}$ on the set where $\psi^-(\rho)$ is nonzero, we have
    \begin{equation} \label{sat-estim-psi}
        \sigma(\rho) \leq c_1 \Theta \left(\dfrac{\omega}{2}\right) \leq  2^{\beta } \dfrac{c_1}{c_0} \dfrac{1}{\theta_1}.
    \end{equation}
    In view of \eqref{t0-def}-\eqref{t0-estim}, \eqref{psi-esim}-\eqref{sat-estim-psi},  
    the logarithmic estimate \eqref{log-energy-estim} reads 
    \begin{align} \label{log-estimate-full-1}
    \int_{B_R \times \{t\}} [\psi^-(\rho)]^2 dx &\leq \int_{B_R \times \{\hat{t} \} } [\psi^-(\rho)]^2 dx + \dfrac{C(\beta,c_0,c_1)}{\theta_1  } \cdot  \iint\limits_{Q(\hat{t},R)} \psi^-(\rho) |\nabla \zeta|^2   dx dt   \nn \\
    &+ \dfrac{C(\sigmam,\Lambda)}{c^2} \left( 1+ \log \dfrac{H_{\rho,\hat{k}}^-}{c} \right) \left[  \int_{\hat{t} }^{0} |A_{\hat{k},R}^-(t)|dt \right]  \nn \\
    &\leq  s_0^2 (\log 2)^2 \left( \dfrac{1-\nu}{1-\frac{\nu}{2}}\right)|B_R|  + \dfrac{C(\beta,c_0,c_1)}{\lambda^2} (s_0 \log 2) |B_R|  \nn \\
    &+ C(\sigmam,\Lambda) (s_0 \log 2) c_0^{-1} 2^{2(\beta+s_0)}  \omega^{-(\beta+2)} R^{\hat{b} N \kappa} |B_R|   \nn \\
    &\leq  \left[ s_0^2 (\log 2)^2 \left( \dfrac{1-\nu}{1-\frac{\nu}{2}}\right) + C(\beta,c_0,c_1, \sigmam,\Lambda) \dfrac{s_0 \log 2}{\lambda^2} \right] |B_R| ,
\end{align}
where $\kappa:=  1 + 4/N $ and $\hat{b}:= 2/(N+4)$. Thanks to \Cref{LOT}, we have assumed without loss of generality that
$$ c_0^{-1} 2^{2(\beta+s_0)}  \omega^{-(\beta+2)} R^{\hat{b} N \kappa} \leq 1,$$ 
leading to the last inequality in \eqref{log-estimate-full-1}. We now estimate the leftmost side of \eqref{log-estimate-full-1} below by integrating $[\psi^-(\rho)]^2$ over a smaller set $S \subset B_R$ defined as 
\begin{equation*}
    S := \left\{x \in B_{(1-\lambda)R} : \rho(x,t) < c \right\}.
\end{equation*} 
To this end, observe that on the set $S$, $\zeta = 1$ and the map 
$$H^-_{\rho,k} \mapsto \dfrac{H^-_{\rho,\hat{k}}}{H^-_{\rho,\hat{k}} + \rho - \hat{k} + c } $$
is decreasing since $\rho - \hat{k} + c <0$. It follows from the definition of the set $S$ that 
\begin{equation*} 
    \dfrac{H^-_{\rho,\hat{k}}}{H^-_{\rho,\hat{k}} + \rho - \hat{k} + c} \geq \dfrac{\tfrac{\omega}{2}}{\tfrac{\omega}{2} + \rho - \hat{k} + c} > \dfrac{\tfrac{\omega}{2}}{\tfrac{\omega}{2^{s_0}}} = 2^{s_0-1}.
\end{equation*}
We thus have the estimate 
\begin{equation*}  
    \psi^-(\rho) \geq (s_0-1) \log 2,
\end{equation*}
and consequently 
\begin{align} \label{log-estimate-full-1-below}
    \int_{B_R \times \{t\}} [\psi^-(\rho)]^2 dx \geq \int_{B_{(1-\lambda)R} \times \{t \}}  [\psi^-(\rho)]^2 dx \geq  (s_0-1)^2 (\log 2)^2 |S|.
\end{align}
Combining $\eqref{log-estimate-full-1}$ with $\eqref{log-estimate-full-1-below}$ and noting that $s_0(s_0-1)^{-2} \leq s_0^{-1}$, we deduce
\begin{equation*} \label{S-set}
    |S| \leq \left[  \left( \dfrac{s_0}{s_0-1} \right)^2  \left( \dfrac{1- \nu}{1-\frac{\nu}{2}} \right) +  \dfrac{C(\beta,c_0,c_1,\sigmam,\Lambda)}{\lambda^2 s_0} \right] |B_R|,
\end{equation*}
and hence for all $t \in (\hat{t},0)$, we have
\begin{align*}
    \left| x \in B_R : \rho(x,t) <  \dfrac{\omega}{2^{s_0+1}}\right| &\leq |S| + |B_R \backslash B_{(1-\lambda)R}| \\
    &\leq |S| + |B_R| \cdot N\lambda \\
    &\leq \left[  \left( \dfrac{s_0}{s_0-1} \right)^2  \left( \dfrac{1- \nu}{1-\frac{\nu}{2}} \right) +  \dfrac{C(\beta,c_0,c_1,\sigmam,\Lambda)}{\lambda^2 s_0} + N\lambda \right] |B_R|.
\end{align*}
The proof concludes by choosing $\lambda$ sufficiently small such that
$$N \lambda \leq \dfrac{3}{8} \nu^2$$
and $s_0$ sufficiently large such that
$$\dfrac{C(\beta,c_0,c_1,\sigmam,\Lambda)}{\lambda^2 s_0} \leq \dfrac{3}{8} \nu^2  \ \ \ \text{and} \ \ \ \left( \dfrac{s_0}{s_0-1} \right)^2 \leq  \left(1-\frac{\nu}{2}\right) \left(1+\nu\right).$$
With these choices, \eqref{conc-1} follows upon taking $s_1 := s_0 + 1$.
    \end{proof}
\end{proposition}
As a consequence of the previous proposition, we have the following estimate of the measure of $\rho$ near 0. 
\begin{corollary} \label{cor-1}
    There exists a number $s_1 \in \mathbb{N}$, depending only on the data, such that for all $t \in \left(- \tfrac{\theta_1}{2} R^2,0 \right)$,
    \begin{equation} \label{conc-2}
        \left| x \in B_R : \rho(x,t) <  \dfrac{\omega}{2^{s_1}}\right| < \Big(1 - \left(\frac{\nu}{2} \right)^2 \Big) |B_R|.
    \end{equation}
\end{corollary}
To simplify the notation in the sequel, we set
\begin{equation} \label{theta2}
    \theta_* := \dfrac{\theta_1}{2}.
\end{equation}

We use Corollary~\ref{cor-1} and the basic energy estimate constructed in \Cref{energy-estim-prop} to deduce that, within a parabolic subcylinder $Q(\theta_* R^2, R) \subset Q(\theta_1 R^2, R)$ and for a given $\nu_* \in (0,1)$, the measure of the set where $\rho$ is near its infimum is further smaller. This reduction is achieved by determining a smaller level (prescribed by a number $s_*>s_1$) below which such a measure can be made arbitrarily small.

\begin{proposition}
\label{propv2}
    For every $\nu_* \in (0,1)$, there exists $s_* \in \mathbb{N}$, $s_* > s_1$, depending only on the data, such that
    \begin{equation} 
        \left| \left\{ (x,t) \in Q\left( \theta_* R^2, R \right) : \rho(x,t) <  \dfrac{\omega}{2^{s_*}} \right\} \right| \leq \nu_* \left| Q\left( \theta_* R^2, R \right) \right|.
    \end{equation}
    \begin{proof}
       Consider the energy estimate \eqref{energy-estim} over the parabolic cylinder $Q(2 \theta_* R^2, 2R)$ for the quantity $(\rho-k_0)_-$, where
\begin{equation*}
    k_0 := \dfrac{\omega}{2^s},
\end{equation*}
for some $s \in \N$ with $s \geq s_1 > 1$ and $s_* \geq s$ to be chosen. Let $0 \leq \zeta(x,t) \leq 1$ smooth piecewise cutoff function in $Q(2 \theta_* R^2, 2R)$ satisfying
 \begin{equation} \label{cutoffs-2}
         \begin{cases}
             \zeta(x,t) \equiv 1, & \text{in}  \ \  Q( \theta_* R^2, R), \\
             \zeta(x,t) \equiv 0, & \text{on}  \ \ \partial_p Q(2 \theta_* R^2, 2R), \\
             0 \leq \partial_t \zeta \leq \dfrac{2}{\theta_1 R^{2}}, & |\nabla \zeta| \leq \dfrac{1}{R}.
         \end{cases}
     \end{equation}
    To begin with, note that the first term on the left-hand side of \eqref{energy-estim} is nonnegative; thus, it may be disregarded. As to the second term therein, we apply assumption $\bf{S2}$ where $\nabla(\rho-k_0)_-$ is nonzero to obtain 
    \begin{align*}
        \sigma(\rho) &\geq c_0 \Theta \left(1 - k_0 \right)  \geq c_0 \Theta \left( \dfrac{\omega}{4} \right) = \dfrac{1}{\theta_1}, \\
        \sigma(\rho) & \leq c_1 \Theta \left(\omega \right) = 2^{2\beta} c_1 \Theta \left(\dfrac{\omega}{4} \right) \leq 2^{2\beta} \dfrac{c_1}{c_0} \dfrac{1}{\theta_1}.
    \end{align*}
    With these bounds and the choice of the cutoff function \eqref{cutoffs-2}, estimate \eqref{energy-estim} reads 
    \begin{align} \label{2R-basic-estimate}
         \dfrac{1}{\theta_1}  &\iint\limits_{Q(\theta_* R^2,R)} 
         |\nabla (\rho - k_0)_-|^2 \zeta^2 dx dt \leq \underbrace{\dfrac{C(\beta,c_0,c_1)}{\theta_1  R^2} \iint\limits_{Q(2\theta_* R^2,2R)} (\rho - k_0)_-^2 dx dt }_{\mathcal{I}_1} \nn \\
         &+ \underbrace{\dfrac{2}{\theta_1 R^2} \iint\limits_{Q(2\theta_* R^2,2R)} (\rho - k_0)_-^2 dx dt}_{\mathcal{I}_2} 
         + \underbrace{C(\sigmam,\Lambda)    \left[\int_{-2\theta_* R^2}^{0} |A^-_{k_0,2R}(t)|dt \right]}_{\mathcal{I}_3}.
    \end{align}
    Observing that $(\rho-k_0)_- \leq \frac{\omega}{2^s}$, we estimate each term
     \begin{align*} 
         \mathcal{I}_1 &\leq \dfrac{C(\beta,c_0,c_1)}{\theta_1 R^2}   \left( \dfrac{\omega}{2^s} \right)^2 2^{N+1} \left| Q(\theta_* R^2, R) \right| \\
         &\leq \dfrac{C(\beta,c_0,c_1,N)}{\theta_1 R^2} \left(\dfrac{\omega}{2^s} \right)^2 \left| Q(\theta_* R^2, R) \right|, \\ 
         \mathcal{I}_2 &\leq  \dfrac{C(N)}{\theta_1 R^2} \left( \dfrac{\omega}{2^s} \right)^2  \left| Q(\theta_* R^2, R) \right|, \\ 
          \mathcal{I}_3 &\leq \dfrac{C(\sigmam,\Lambda,N)}{\theta_1 R^2} \left( \dfrac{\omega}{2^s} \right)^2  \left| Q(\theta_* R^2, R) \right| c_0^{-1} 2^{2(\beta+s_*)} \omega^{-(\beta+2)} R^{\hat{b}N\kappa}, \nn  
     \end{align*}
where $\kappa: =1 + 4/N$ and $\hat{b}:= 2/(N+4)$. By \Cref{LOT}, we may estimate 
$$c_0^{-1} 2^{2(\beta+s_*)} \omega^{-(\beta+2)} R^{\hat{b}N\kappa} \leq 1.$$
Consequently, \eqref{2R-basic-estimate} reads
     \begin{align} \label{2R-final-estimate}
        \iint\limits_{Q(\theta_* R^2,R)} 
         |\nabla (\rho - k_0)_-|^2  dx dt \leq  \dfrac{C(\beta,c_0,c_1,\sigmam,\Lambda,N)}{R^2} \left( \dfrac{\omega}{2^s} \right)^2  \left| Q(\theta_* R^2, R) \right|.
    \end{align}
Now, consider another level
\begin{equation*}
    k_1 :=  \dfrac{\omega}{2^{s+1}} < k_0,
\end{equation*}
and define, for $t \in(-\theta_* R^2,0)$, the set
\begin{equation*}\label{At}
    A_s(t) := \left \{ x \in B_R: \rho(x,t) < k_0 \right\},
\end{equation*}
along with its measure
\begin{equation*}
    |A_s| = \int_{-\theta_* R^2}^0 |A_s(t)| dt.
\end{equation*}
Applying \Cref{DegLemma} to $\rho(\cdot,t)$ with the preceding choices for $k_0$ and $k_1$, we obtain
\begin{equation} \label{degiorgi-1}
    \left( \dfrac{\omega}{2^{s+1}} \right)|A_{s+1}(t)| \leq C(N) \dfrac{R^{N+1}}{|B_R \backslash A_s(t)|} \int_{A_s(t)\backslash A_{s+1}(t)} |\nabla \rho| dx.
\end{equation}
    Owing to the observation
    \begin{equation} \label{s-cond}
        \frac{\omega}{2^{s}} \leq \frac{\omega}{2^{s_1}} < \frac{\omega}{2},
    \end{equation}
    and \Cref{cor-1} we have that, for $t \in (-\theta_* R^2,0)$,
    $$|A_s(t)| \leq |A_{s_{_{1}}}(t)| < \Big(1 - \left(\frac{\nu}{2} \right)^2 \Big) |B_R|,$$
    which implies
    $$|B_R \backslash A_s(t)| \geq \left(\dfrac{\nu}{2} \right)^2,$$
    and thus \eqref{degiorgi-1} becomes 
    \begin{equation} \label{degiorgi-2}
    \left( \dfrac{\omega}{2^{s+1}} \right)|A_{s+1}(t)| \leq  \dfrac{C(N)}{\nu^2} R \int_{A_s(t)\backslash A_{s+1}(t)}  |\nabla \rho| dx.
\end{equation}
Integrating \eqref{degiorgi-2} in time over the interval $(-\theta_* R^2,0)$ and applying H\"older's inequality yield
\begin{align} \label{As1}
    \dfrac{1}{2} \left( \dfrac{\omega}{2^{s}} \right)|A_{s+1}| & \leq  \dfrac{C(N)}{\nu^2} R \int_{-\theta_* R^2}^0 \int_{A_s \backslash A_{s+1} } |\nabla \rho| dx \nn \\
     & \leq  \dfrac{C(N)}{\nu^2} R \left( \ \iint\limits_{A_s} 
         |\nabla (\rho - k_0)_-|^2  dx dt  \right)^{\frac{1}{2}} |A_s \backslash A_{s+1}|^{\frac{1}{2}}.
\end{align}
Estimating the integral on the right hand side of \eqref{As1} by \eqref{2R-final-estimate}, we get
\begin{equation} \label{As12}
    |A_{s+1}|^2 \leq \dfrac{C(\beta,c_0,c_1,\sigmam,\Lambda,N)}{\nu^{4}} \left| Q(\theta_* R^2, R) \right|  |A_s \backslash A_{s+1}|,
\end{equation}
valid for $1 < s_1 \leq s \leq s_*$. Summing \eqref{As12} over $s = s_1,s_1+1, \dots, s_*-1$ along with minorizing and majorizing the obtained series, we get
\begin{align*}
    (s_*-s_1)|A_{s_*}|^2 &\leq \sum_{s = s_1}^{s_*-1} |A_{s+1}|^2 \\
    &= \dfrac{C(\beta,c_0,c_1,\sigmam,\Lambda,N)}{\nu^{4}}  \left| Q(\theta_* R^2, R) \right|  \sum_{s = s_1}^{s_*-1} |A_s \backslash A_{s+1}| \\
    & \leq \dfrac{C(\beta,c_0,c_1,\sigmam,\Lambda,N)}{\nu^{4}}  \left| Q(\theta_* R^2, R) \right|^2  ,
\end{align*}
where we have used  $|A_{s+1}| \geq |A_{s_*}|$ and $\sum_{s = s_1}^{s_*-1} |A_s \backslash A_{s+1}| \leq |Q(\theta_*R^2,R)|$. We finally arrive at the estimate
\begin{equation*}
    |A_{s_*}|^2 \leq \dfrac{C(\beta,c_0,c_1,\sigmam,\Lambda,N)}{\nu^{4} (s_*-s_1)} \left| Q(\theta_* R^2, R) \right|^2.
\end{equation*}
The proof is complete once a sufficiently large $s_* \in \mathbb{N}$, $s_*>s_1$, is chosen such that
\begin{equation*}
    \nu_* \geq \dfrac{1}{\nu^2} \left( \dfrac{C(\beta,c_0,c_1,\sigmam,\Lambda,N)}{s_*-s_1}\right)^{\frac{1}{2}}.
\end{equation*}
\end{proof}
\end{proposition}

We conclude by showing that $\rho$ is strictly above its infimum value in a proper, coaxial parabolic subcylinder of $Q(\theta_* R^2, R)$. This assertion is employed subsequently to demonstrate the oscillation reduction when \eqref{second-alt} holds.

\begin{lemma} \label{lemma2}
    The number $\nu_*$ in \Cref{propv2}, and consequently the number $s_*$, can be chosen such that 
\begin{equation}  \label{conc-3}
        \rho(x,t) \geq \dfrac{\omega}{2^{s_*+1}}, \ \ \text{a.e. in }  \ \  Q\bigg(\theta_* \left(\frac{R}{2}\right)^2,\dfrac{R}{2}\bigg).
\end{equation}
\begin{proof}
    The reasoning here is analogous to the proof of \Cref{prop1}, so we will outline the pertinent differences and dispense with the details. We work with quantity 
    \begin{equation*}
        \rho_\omega := \max \left\{\rho , \dfrac{\omega}{2^{s_*+1}} \right\},
    \end{equation*}
    and, for $n = 0,1,2,\dots$, the sequences 
    \begin{align*}
        R_n &= \dfrac{R}{2} + \dfrac{R}{2^{n+1}},  \\
        k_n & =  \dfrac{\omega}{2^{s_*+1}} +  \dfrac{\omega}{2^{s_*+1+n}} .
    \end{align*}
     Construct a family of nested parabolic cylinders $Q_n^* := Q(\theta_* R^2_n,R_n)$, along with corresponding cutoff functions $0 \leq \zeta_n \leq 1$ in $Q_n^*$ satisfying 
     \begin{equation*} \label{cutoffs-3}
         \begin{cases}
             \zeta_n \equiv 1 \ \text{in} \ Q_{n+1}^*, \\ 
            \zeta_n \equiv 0  \ \text{on} \  \partial_p Q_n^*, \\
              0 < \partial_t \zeta_n \leq \dfrac{2^{2(n+2)}}{\theta_* R^2}, \ \ \  |\nabla \zeta_n| \leq \dfrac{2^{n+2}}{R}, \ \ |\Delta \zeta_n| \leq \dfrac{2^{2(n+2)}}{ R^2}.
         \end{cases}
     \end{equation*}
    If $k_n > \rho_\omega = \rho > \dfrac{\omega}{2^{s_*+1}}$, then
    \begin{align*}
        (\rho_{\omega} -k_n)_- =(k_n - \rho_{\omega})_+ &\leq \left( \dfrac{\omega}{2^{s_*+1+n}} \right) \leq  \left( \dfrac{\omega}{2^{s_*+1}} \right).
    \end{align*}   
Alternatively, if $0 \leq \rho \leq \dfrac{\omega}{2^{s_*+1}} = \rho_\omega < k_n$,
    \begin{align*}
        (\rho_\omega - k_n)_- &=   \left( \dfrac{\omega}{2^{s_*+1+n}} \right) \leq  \left( \dfrac{\omega}{2^{s_*+1}} \right).
    \end{align*}
Note also that $\nabla (\rho_\omega - k_n)_-$ is nonzero on the set $\left\{0\leq  \dfrac{\omega}{2^{s_*+1}} < \rho < k_n \leq \dfrac{\omega}{2^{s_*}} \right\}$. Moreover, on this very set, assumption $\bf{S2}$ yields the bounds 
\begin{align*} \label{sat-lower-bnd}
    \sigma(\rho) &\geq c_0 \Theta\left(1-\dfrac{\omega}{4} \right) \geq c_0 \Theta\left(\dfrac{\omega}{4} \right) = \dfrac{1}{\theta_1}, \\ 
     \sigma(\rho) &\leq c_1 \Theta\left(\omega \right) = 2^{2\beta} \dfrac{c_1}{c_0} \dfrac{1}{\theta_1}. 
\end{align*}
We define the quantity 
\begin{equation*}
    \Sigma_+(\rho) := \left[\int_{\rho}^{\tfrac{\omega}{2^{s_*+1}}} \sigma(s) ds \right]_+, 
\end{equation*}
satisfying, on the set $\left\{ 0 \leq \rho \leq \dfrac{\omega}{2^{s_*+1}}  \right\}$,  
\begin{equation*}
    \Big|\Sigma_+(\rho)\Big| \leq   c_1 \Theta\left(1-\rho \right)   \left|\dfrac{\omega}{2^{s_*+1}} - \rho \right|  \leq 2^{2\beta} \dfrac{c_1}{c_0} \dfrac{1}{\theta_1} \left( \dfrac{\omega}{2^{s_*+1}} \right).
\end{equation*}
We proceed to construct an energy estimate for $(\rho_\omega-k_n)_-$. In \eqref{weak_form_stek}, take 
    \begin{equation*}
        \phi = -\left( (\rho_{\omega})_h - k_n \right)_- \zeta_n^2,
    \end{equation*}
and integrate over the time interval $(-\theta_* R^2_n,t_\circ)$, for $t_\circ \in (-\theta_* R^2_n,0)$:
        \begin{align} \label{term1-2}
    &\underbrace{\iint\limits_{Q_n^*} \dt \rho_h \left[ -\left( (\rho_{\omega})_h - k_n \right)_- \zeta_n^2 \right] dx dt}_{\mathcal{I}_1} + \underbrace{\iint\limits_{Q_n^*}  (\sigma(\rho) \nabla \rho)_h \cdot \nabla \left[ - \left( (\rho_{\omega})_h - k_n \right)_- \zeta_n^2 \right] dx dt}_{\mathcal{I}_2} \nn \\
        &+\underbrace{ \iint\limits_{Q_n^*} (\rho \sigma(\rho))_h \nabla \big[  V +  W*\rho_h \big] \cdot \nabla \big[-\left( (\rho_{\omega})_h - k_n \right)_- \zeta_n^2 \big] dx dt}_{\mathcal{I}_3} = 0.
    \end{align}
After sending $h \to 0$, the first term can be estimated as in the proof of \Cref{prop1} (with the obvious modifications):
\begin{align} \label{estim-I-2}
    \mathcal{I}_1 &\geq \dfrac{1}{4}\int_{B_n \times \{t_\circ\}} (\rho_{\omega} - k_n)_-^2 \zeta_n^2 dx - 3 \left(\dfrac{\omega}{2^{s_*+1}}\right)^2 \dfrac{2^{2(n+2)}}{\theta_* R^2} \iint\limits_{Q_n^*}  \chi_{\{ \rho_{\omega} < k_n\}}dx dt.
\end{align}
As to the second term, we have
\begin{align*}
    \mathcal{I}_2 &\to \iint\limits_{Q_n^*} \sigma(\rho_\omega) | \nabla \left( \rho_{\omega}  - k_n \right)_-|^2 \zeta_n^2 dx dt \nn \\
        &+ 2 \iint\limits_{Q_n^*} \sigma(\rho_\omega) \nabla\left( \rho_{\omega}  - k_n \right)_- \cdot \nabla \zeta_n \zeta_n \left( \rho_{\omega}  - k_n \right)_- dx dt \nn \\
        &+ \underbrace{2 \left(\dfrac{\omega}{2^{s_*+1}}\right) \iint\limits_{Q_n^*} -\sigma(\rho) \nabla \rho \cdot \nabla \zeta_n \zeta_n  \chi_{\left\{ \rho \leq \tfrac{\omega}{2^{s_*+1}}\right\}
        } dx dt}_{\mathcal{I}'_2}.
\end{align*}
Applying Young's inequality to the second term in the preceding identity, we obtain
\begin{align*}  
        &\left|  2 \iint\limits_{Q_n^*} \sigma(\rho_\omega) \nabla\left( \rho_{\omega}  - k_n \right)_- \cdot \nabla \zeta_n \zeta_n \left( \rho_{\omega}  - k_n \right)_- dx dt \right| \nn \\
        &\leq \varepsilon \iint\limits_{Q_n^*}   \sigma(\rho_\omega) | \nabla \left( \rho_{\omega}  - k_n \right)_-|^2 \zeta_n^2 dx dt + \dfrac{1}{ \varepsilon} \iint\limits_{Q_n^*} \sigma(\rho_\omega) \left( \rho_{\omega}  - k_n \right)^2_- |\nabla \zeta_n|^2 dx dt,
    \end{align*}
and estimating the third term in that same identity from above as
\begin{align*}  
        \left| \mathcal{I}'_2 \right| &\leq \left| \left(\dfrac{\omega}{2^{s_*+1}}\right) \iint\limits_{Q_n^*} \nabla\Sigma_+(\rho) \nabla \rho \cdot \nabla \zeta_n \zeta_n  
        dx dt \right| \nn \\
        &= \left| 2 \left(\dfrac{\omega}{2^{s_*+1}}\right) \iint\limits_{Q_n^*}  \Sigma_+(\rho)  \left( \zeta_n |\Delta \zeta_n| + |\nabla \zeta_n|^2 \right) 
         dx dt \right| \nn \\
        &  \leq 4 \cdot 2^{2\beta}  \dfrac{c_1}{c_0} \dfrac{1}{\theta_1} \left(\dfrac{\omega}{2^{s_*+1}} \right)^2 \bigg(\dfrac{2^{2(n+2)}}{ R^2}\bigg) \iint\limits_{Q_n^*}  \chi_{\{ \rho_{\omega} < k_n\}}  dx dt .
    \end{align*}
    Choosing $\varepsilon= 1/4$ and estimating as delineated earlier, we obtain the following lower bound:
    \begin{align} \label{estim-II-2}
        \mathcal{I}_2 &\geq \dfrac{3}{4} \dfrac{1}{\theta_1} \iint\limits_{Q_n^*} | \nabla \left( \rho_{\omega}  - k_n \right)_-|^2 \zeta_n^2 dx dt  \nn \\
        &- 4 (1+ 2^{2\beta}) \left(\dfrac{\omega}{2^{s_*+1}}\right)^2  \bigg(\dfrac{2^{2(n+2)}}{\theta_1 R^2}\bigg) \iint\limits_{Q_n^*}  \chi_{\{ \rho_\omega < k_n\}} dx dt.
    \end{align}
    As to the third term in \eqref{term1-2}, we have  
     \begin{align} \label{estim-III-2}
        |\mathcal{I}_3| 
        & \leq  \dfrac{\sigmam \Lambda}{2 \varepsilon_1} \iint\limits_{Q_n^*} | \nabla \left( \rho_{\omega}  - k_n \right)_-|^2 \zeta_n^2 dx dt +  \dfrac{\sigmam \Lambda}{ \varepsilon_2} \iint\limits_{Q_n^*} \left( \rho_{\omega}  - k_n \right)^2_-  |\nabla \zeta_n|^2 dx dt \nn \\
        &+ \sigmam \Lambda \left( \dfrac{\varepsilon_1}{2}  + \varepsilon_2    \right) \iint\limits_{Q_n^*} \chi_{\{ \rho_{\omega} < k_n\}}  dx dt \nn  \\
        & \leq  \dfrac{1}{4} \dfrac{1}{\theta_1}\iint\limits_{Q_n^*} | \nabla \left( \rho_{\omega}  - k_n \right)_-|^2 \zeta_n^2 dx dt \nn \\
        &+  2^{2\beta} \dfrac{c_1}{c_0}  \left(\dfrac{\omega}{2^{s_*+1}}\right)^2  \bigg(\dfrac{2^{2(n+2)}}{\theta_1 R^2}\bigg) \iint\limits_{Q_n^*}  \chi_{\{ \rho_\omega < k_n\}} dx dt \nn \\
        &+ 2(\sigmam \Lambda)^2  \theta_1  \left[ \int_{-\theta_* R^2_n }^{0} |A_{k_n,R_n}^-(t)|dt \right],
        \end{align}
    with the choices $ \varepsilon_1 :=  2\sigmam \Lambda  \theta_1$, $\varepsilon_2 = \sigmam \Lambda \theta_1c_0/c_1$, and the set
        $$A_{k_n,R_n}^-(t) := \{ x \in B_n: \rho_\omega(x,t) < k_n\},$$
        with the shorthand $B_n:= B_{R_n}$. Recalling that $\theta_1 \geq \theta_*$ and collecting  \eqref{estim-I-2}, \eqref{estim-II-2}, and \eqref{estim-III-2}, we arrive at the energy estimate
\begin{align}\label{full-estim-2}
            \esssup_{-\theta_* R_n^2 < t < 0} &\int_{B_n \times \{t\}} (\rho_\omega - k_n)_-^2 \zeta_n^2 dx + \dfrac{2}{\theta_1} \iint\limits_{Q_n^*} 
         |\nabla (\rho_{\omega}  - k_n)_-|^2 \zeta_n^2 dx dt \nn \\
          &\leq C(\beta,c_0,c_1) \left(\dfrac{\omega}{2^{s_*+1}}\right)^2  \bigg(\dfrac{2^{2(n+2)}}{\theta_* R^2}\bigg) \iint\limits_{Q_n^*}  \chi_{\{ \rho_\omega < k_n\}} dx dt \nn \\
          &+ C(\sigmam,\Lambda) \theta_1 \theta_*  \left[ \dfrac{1}{\theta_*} \int_{-\theta_* R^2_n }^{0} |A_{k_n,R_n}^-(t)|dt \right].
        \end{align}
    Notice that by \eqref{scl} we have 
\begin{equation*}
    \theta_1 \theta_* \leq \theta_1^{2} = c_0^{-2}\left(\dfrac{\omega}{4}\right)^{-2\beta}
    \leq c_0^{-2} 2^{4\beta + 2s_*} \omega^{-2\left(\beta+1\right)} \left(\dfrac{\omega}{2^{s_*+1}}\right)^2 
    \leq    R^{-N \kappa} \left(\dfrac{\omega}{2^{s_*+1}}\right)^2 ,
\end{equation*}
where we have assumed that
\begin{equation*}
    c_0^{-2} 2^{4\beta + 2s_*} \omega^{-2\left(\beta+1\right)} R^{N\kappa} \leq 1,
\end{equation*}
with $\kappa := 1 + 4/N$. Applying the change of the time variable
\begin{equation*}\label{change-of-var-2}
    \tilde{t} = \dfrac{t}{\theta_*},
\end{equation*}
the estimate $\eqref{full-estim-2}$  transforms to
\begin{align*}
    \|( \tilde{\rho}_\omega - k_n)_- \tilde{\zeta}_n\|_{V^2(Q_n)} &\leq  C(\beta,c_0,c_1) \left(\dfrac{\omega}{2^{s_*+1}}\right)^2  \bigg(\dfrac{2^{2(n+2)}}{R^2}\bigg) |\tilde{A}^-_n| \nn \\
          &+ C(\sigmam,\Lambda) R^{-N\kappa} \left(\dfrac{\omega}{2^{s_*+1}}\right)^2  \left[\int_{-R^2_n }^{0} |\tilde{A}^-_n( \tilde{t} )|d\tilde{t} \right],
    \end{align*}
       where $Q_n := Q(R_n^2,R_n)$ and
        \begin{equation*}
            \tilde{A}^-_n(\tilde{t}) := \{ x \in B_n : \tilde{\rho}_\omega(x,\tilde{t}) < k_n \}, \ \ \ \text{and} \ \ \
            |\tilde{A}^-_n| = \int_{-R^2_n }^{0} |\tilde{A}^-_n( \tilde{t} )|d\tilde{t}. 
        \end{equation*}
    Letting 
\begin{equation*}
    Y_n := \dfrac{|\tilde{A}^-_n|}{|Q_n|}, \ \ \ \ \ \ Z_{n} := \dfrac{|\tilde{A}^-_n|^{\frac{1}{1+\kappa}}}{|B_n|},
\end{equation*}
and following a similar reasoning employed in the proof of \Cref{prop1}, we get the recursive inequalities:
\begin{align} \label{recur3}
\begin{cases}
        Y_n \leq C(\beta,c_0,c_1,\sigmam,\Lambda,N) 4^{2n} \big( Y_n^{1+ \frac{2}{N+2}} + Y_n^{\frac{2}{N+2}} Z_n^{1+ \kappa} \big),\\
    Z_n \leq C(\beta,c_0,c_1,\sigmam,\Lambda,N) 4^{2n} \big(   Y_n +  Z^{1+\kappa}_{n}\big).
\end{cases}
\end{align}
Relying again on \Cref{fast-conv-lemma}, if \eqref{recur3} holds and 
\begin{equation} \label{conv-cond-2}
    Y_0 + Z_0^{1+\kappa} \leq (2 C(\beta,c_0,c_1,\sigmam,\Lambda,N))^{-\tfrac{1+\kappa}{d}}4^{-\tfrac{2(1+\kappa)}{d^2}}=: \nu_* \in (0,1),
\end{equation}
where $d := \min\{\kappa,2/(N+2) \} = 2/(N+2)$, it follows that $Y_n, Z_n \to 0$ as $n \to \infty$. Fix $\nu_*$ as in \eqref{conv-cond-2} and choose $s_*$ 
according to \Cref{propv2} so that $Y_0 \leq \nu_*$ holds and \eqref{conc-3} follows. 
\end{proof}
\end{lemma}

We now derive a result parallel to Corollary~\ref{oscRed1}, but when \eqref{second-alt} holds. Recall that $\nu$ has been fixed in \eqref{conv-cond-1} and $\theta_*$ is given in \eqref{theta2}.

\begin{corollary} \label{oscRed2} Assume \eqref{assump1} and \eqref{second-alt} hold. Then, there exists a constant $\gamma_2 \in (0,1)$, depending only on the data, such that 
    \begin{equation*} 
   \essosc_{ Q\left(\theta_* \left(\frac{R}{2}\right)^2, \frac{R}{2}\right)} \rho \leq  \gamma_2 \omega.
\end{equation*}
\begin{proof}
    By \Cref{lemma2} and an argument similar to the proof of Corollary~\ref{oscRed1}, we reach to the stated conclusion with
    \begin{equation*}
        \gamma_2 := 1 - \dfrac{1}{2^{s_*+1}},
    \end{equation*}
    where $s_*$ is dependent only on the data per the preceding Lemma.
\end{proof}
\end{corollary}

 \subsection{Oscillation decay} \label{recursive-sec}
As a consequence of the established oscillation reduction in Corollaries~\ref{oscRed1}-\ref{oscRed2} for the two complementary cases \eqref{first-alt}-\eqref{second-alt}, $\nu \in (0,1)$ determined in \eqref{conv-cond-1},  and the observation  
\begin{equation*}
    Q\bigg(\theta_* \left(\dfrac{R}{2} \right)^2, \dfrac{R}{2} \bigg) \subset  Q\bigg(\theta_1 \left(\dfrac{R}{2} \right)^2, \dfrac{R}{2} \bigg),
\end{equation*}
we obtain, upon fixing
\begin{equation} \label{small-gam}
    \gamma := \max\{\gamma_1,\gamma_2\} = \gamma_2,
\end{equation}
for $s_* > 3$, the following assertion.
\begin{corollary} [Oscillation reduction] \label{oscRed}
    There exist constants $\nu,\gamma \in (0,1)$, depending only on the data,  such that
    \begin{equation*} 
   \essosc_{ Q\left( \theta_* \left(\frac{R}{2}\right)^2, \frac{R}{2}\right)} \rho \leq  \gamma \omega.
\end{equation*}
\end{corollary}

\begin{lemma} 
    [Induction] \label{Induc-lemma} For a positive constant $\eta < 1$ to be fixed depending only on the data and $\gamma<1$ fixed in \eqref{small-gam}, we can construct, for $k \in \N$, the sequences:
    \begin{equation*}
        R_k:= \eta^{k-1} R, \ \ \ \ \omega_k : = \gamma^{k-1} \omega, \ \ \ \ \dfrac{1}{\theta_k} = c_0 \Theta \left(\dfrac{\omega_k}{4} \right), \ \ \ \ Q_k:= Q(\theta_k R_k^2,R_k),
    \end{equation*}
    such that
    \begin{equation} \label{induction-result-k}
        Q_{k+1} \subset Q_k, \ \ \ \ \text{and} \ \ \ \ \essosc_{Q_k} \rho \leq {\omega_k}.
    \end{equation}
 
\end{lemma}

\begin{proof}
    Recalling \eqref{scl} and the initial assumption \eqref{assump2}, we readily have:
    \begin{equation*}
        R_1 = R, \ \ \ \, \omega_1 = \omega, \ \ \ \  \dfrac{1}{\theta_1} = c_0 \Theta \left(\dfrac{\omega}{4} \right), \ \ \ \ \text{and} \ \ \ \ \essosc_{Q_1} \rho \leq {\omega}.
    \end{equation*}
    We observe that
    \begin{align*}
        \theta_* \left(\frac{R}{2}\right)^2 = \dfrac{\theta_1}{2} \left(\frac{R}{2}\right)^2 &= \dfrac{c_0^{-1}}{2} \left(\dfrac{\omega}{4} \right)^{-\beta} \left(\dfrac{\omega_2}{4} \right)^{-\beta} \left(\dfrac{4}{\omega_2} \right)^{-\beta} \left(\frac{R}{2}\right)^2 \\ 
        &= \theta_2 2^{-3} R_1^2    \left( \dfrac{\omega_1}{\omega_2} \right)^{-\beta}  \\
        &= \theta_2 R_2^2,
    \end{align*}
    where 
    $R_2:= \eta R_1$ with
    \begin{equation} \label{eta-choice}
        \eta := 2^{-3/2} \gamma^{\beta/2},
    \end{equation}
    and $\beta>0$ as fixed in assumption $\bf{S2}$. Hence $Q_2 \subset Q_1$. Applying \Cref{oscRed}, we obtain
    \begin{equation*} 
   \essosc_{Q_2} \rho \leq \essosc_{ Q\left( \tfrac{\theta_1}{2} \left(\tfrac{R}{2}\right)^2, \tfrac{R}{2}\right)} \rho \leq  \gamma \omega = \omega_2.
\end{equation*}
Repeating the process, we inductively deduce \eqref{induction-result-k}.
\end{proof}

\begin{proposition} [Oscillation decay]\label{oscil-decay}
     There exist constants $\Gamma >1$ and $\alpha \in (0,1)$, that can be determined a priori in terms of the data, such that for all cylinders of the form
    \begin{equation*}
        Q(\theta_1 r^2, r), \ \ \ \text{with} \ \ 0 < r \leq R,
    \end{equation*}
    it holds that
    \begin{equation*}
        \essosc_{Q(\theta_1 r^2, r)} \rho \leq \Gamma \omega \left( \dfrac{r}{R} \right)^{\alpha}.
    \end{equation*}
\end{proposition}
\begin{proof}
    Fixing such an $r$, there exists $k \in \N$ such that
    \begin{equation*}
        \eta^{k} R \leq r < \eta^{k-1} R.
    \end{equation*}
    Choosing $\alpha := \log \gamma / \log \eta \in (0,1)$, where $\eta$ is the quantity in \eqref{eta-choice} of \Cref{Induc-lemma}, we get
    \begin{equation*}
        \gamma^{\tfrac{k}{\alpha}} < \gamma^k \leq \left( \dfrac{r}{R} \right),
    \end{equation*}
and thus 
    \begin{equation*}
         \gamma^k \leq \left( \dfrac{r}{R} \right)^{\alpha}.
    \end{equation*}
    It follows that, upon fixing $\Gamma := \gamma^{-1}$,
    \begin{equation*}
        \omega_k : = \gamma^{k-1} \omega \leq \Gamma \omega \left( \dfrac{r}{R} \right)^{\alpha}.
    \end{equation*}
    As $r < R_k$ and $Q(\theta_1 r^2,r) \subset Q_k$, the claim follows.
\end{proof}

\section{Energy dissipation and asymptotics} \label{stationary-sec}
Problem \eqref{PDE} can be viewed as the time evolution of the energy functional 
\begin{equation} \label{FreeEnergy}
    \F[\rho] = \int_\Omega \rho(\log \rho-1) + \rho V + \dfrac{1}{2} \rho (W * \rho) dx,
\end{equation}
on the admissible set $\{ \rho \in L^1(\Omega): 0 \leq \rho \leq \rho_{\text{max}} \}$. The minimization problem for \eqref{FreeEnergy} and long-time behaviors of solutions to \eqref{PDE} have recently been studied in \cite{carrillo2024aggregation} under the assumptions that $V \in C^2(\overline{\Omega})$ and $W=0$.
Here, we take advantage of the established regularity in the previous section to demonstrate the uniform convergence of such a solution to a stationary limit.



 Upon formally differentiating the free energy \eqref{FreeEnergy}, with respect to time, along smooth solutions of \eqref{PDE}, we obtain 
\begin{align} \label{dEdt}
    \dfrac{d}{dt}\F[\rho] &= \int_\Omega \left(\log \rho  + V + W* \rho \right) \partial_t \rho dx  \nn \\
    &= - \int_\Omega 
    \rho \sigma(\rho) \left| \nabla \left( \log \rho + V + W* \rho \right)  \right|^2 dx  \leq 0,
\end{align}
where we have integrated by parts and applied the no-flux boundary conditions in \eqref{PDE}. It follows that the map $t \mapsto \F[\rho(\cdot,t)]$ is monotone non-increasing. Integrating \eqref{dEdt} in time over the interval $0 < t_0 < t_1$ allows us to define the \emph{dissipation} functional:
\begin{equation} \label{Dissip}
    \D[\rho]: = \int_{t_0}^{t_1}  \int_\Omega 
    \rho \sigma(\rho) \left| \nabla \left( \log \rho + V + W * \rho \right)  \right|^2 dx dt.
\end{equation}
Evidently, the functional \eqref{Dissip} is nonnegative. It is also lower semicontinuous with respect to $L^1$ convergence (\cite[Lemma 3.9]{carrillo2024aggregation}). Moreover, weak solutions to \eqref{PDE} constructed in \cite{carrillo2024aggregation} satisfies the following dissipation inequality (\cite[Sec. 3] {carrillo2024aggregation}) 
\begin{equation*}
     \int_{t_0}^{t_1} \| \partial_t\rho\|^2_{W^{-1,1}(\Omega)} \leq C \Big( \F[\rho(\cdot,t_0)] - \F[\rho(\cdot,t_1)] \Big),
\end{equation*}
for some constant $C>0$ that depends on the data of the problem and an $L^1$-contraction principle (\cite[Sec. 2.1]{carrillo2024aggregation}), i.e., for two admissible initial data $\rho_0$ and $\tilde{\rho}_0$ and corresponding weak solutions $\rho$ and $\tilde{\rho}$, it holds that
\begin{equation*}
    \|\rho(\cdot,t)-\tilde{\rho}(\cdot,t) \|_{L^1(\Omega)} \leq \|\rho_0-\tilde{\rho}_0 \|_{L^1(\Omega)} \ \ \ \text{for} \ \ \ t>0.
\end{equation*}
 With these preliminaries, we proceed to show the following.
\begin{proposition} \label{ConvToSS}
    Suppose $\rho$ is a weak solution to \eqref{PDE} in the sense of \cite{carrillo2024aggregation} and the potentials $V,W \in C^2(\overline{\Omega})$. Then, for any diverging sequence $\{t_n\}_{n \in \N}$, there exists a subsequence $\{t_{j}\}_{j \in \N}$ and a limit $\rho_*$ such that $\rho(\cdot,t_j) \to \rho_*$ uniformly in $\Omega$ as $j \to \infty$. Furthermore, $\rho_*$ is a weak stationary solution.
\end{proposition}
 \begin{proof}
   By \Cref{bndry}, we may work in $\Omega$. Define $\rho_n(x,t):= \rho(x,t_n+t)$ in $\Omega \times [0,1]$. By \Cref{inter-Holder}, for any $x_1, x_2 \in \Omega$ and  $\tau_1,\tau_2 \in [0,1] $, we have
     \begin{align*}
         |\rho_n(x_1,\tau_1) - \rho_n(x_2,\tau_2)| &= |\rho(x_1,t_n+\tau_1) - \rho(x_2,t_n+\tau_2)| \\
         &\leq C \left( |x_1-x_2| + |\tau_1-\tau_2| \right)^{ \frac{\alpha}{2}},
     \end{align*}
     for some $C>0$ independent of $n$ and the choices of $(x_1,\tau_1)$ and $(x_2,\tau_2)$, showing that the sequence $\{ \rho_n \}_{n \in \mathbb{N} }$ is uniformly equicontinuous with respect to $n$. Applying Arzel\`a-Ascoli Theorem, we find a subsequence, denoted by $\{ \rho_j \}_{j \in \N}$, and a continuous function $\rho_*(x,t)$ such that as $j \to \infty$,
    \begin{equation} \label{unif-conv}
        \rho_{j} \to \rho_* \ \ \ \text{uniformly in } \Omega \times [0,1].
    \end{equation}
Since 
$$\lim_{t\to\infty} \int_0^1 \| \partial_t\rho_j(x,t)\|^2_{W^{-1,1}(\Omega)} = 0,$$
we may conclude that $\rho_*$ is stationary, i.e.,  $\rho_*(x,t) = \rho_*(x)$. Finally, using the nonnegativity and lower semicontinuity of $\D[\cdot]$, as well as the monotonicity and decay of $\F[\cdot]$ in time to a stationary value, we obtain
\begin{align*}
    0 \leq  \D[\rho_*(x)] & \leq \liminf_{j \to \infty} \D[\rho_j] \\
    &=  \lim_{j \to \infty} \inf_{k \geq j}\Big( \F[\rho(\cdot,t_k)] - \F[\rho(\cdot,t_k+1)]  \Big) =0.
\end{align*}
It follows that $\D[\rho_*] = 0$ and hence $\rho_*$ is a weak stationary solution of \eqref{PDE}. 
\end{proof}

In view of the preceding proposition, we conclude with the proof of \Cref{conv-thm}.
\begin{proof}[Proof of Theorem \ref{conv-thm}] 
    Since $W=0$, the weak solution $\rho$ is global-in-time and $L^1$-contractive; hence, it makes sense to speak of a stationary state. By \Cref{ConvToSS}, for an arbitrary sequence $\{t_n\}_{n \in \N} \to \infty$, we may find a diverging subsequence $\{t_j \}_{j \in \N}$ along which $\rho_j \to \rho_\infty$ uniformly in $\Omega$, $\rho_\infty$ being a stationary-in-time and $L^1$-contractive (\cite[Theorem 2.8]{carrillo2024aggregation}) weak solution for $\eqref{PDE}$. Thus, the limit is unique, and we immediately have
 \begin{equation*}
        \lim_{t\to \infty} \| \rho(\cdot,t) - \rho_\infty \|_{L^\infty(\Omega)} = 0.
    \end{equation*}
\end{proof}

 \subsection*{Acknowledgements}
This work is part of the author's PhD thesis, supported by funding from King Abdullah University of Science and Technology (KAUST). The author thanks his thesis advisor, Professor Jos\'e Miguel Urbano, for proposing the problem, introducing him to the intrinsic scaling method, and providing valuable feedback. Thanks are also extended to Professor José Antonio Carrillo for bringing the topic to our attention and the referees for their constructive comments and helpful suggestions.

\bibliographystyle{acm}
\bibliography{refs}

\begin{thebibliography}{10}

\bibitem{bailo2024aggregation}
{\sc Bailo, R., Carrillo, J.~A., and G{\'o}mez-Castro, D.}
\newblock Aggregation-diffusion equations for collective behaviour in the sciences.
\newblock {\em arXiv preprint arXiv:2405.16679\/} (2024).

\bibitem{bailo2023boundpreserving}
{\sc Bailo, R., Carrillo, J.~A., and Hu, J.}
\newblock Bound-preserving finite-volume schemes for systems of continuity equations with saturation.
\newblock {\em SIAM J. Appl. Math. 83}, 3 (2023), 1315--1339.

\bibitem{burger2006keller}
{\sc Burger, M., Di~Francesco, M., and Dolak-Struss, Y.}
\newblock The {K}eller--{S}egel model for chemotaxis with prevention of overcrowding: {L}inear vs. nonlinear diffusion.
\newblock {\em SIAM Journal on Mathematical Analysis 38}, 4 (2006), 1288--1315.

\bibitem{carrillo2019aggregation}
{\sc Carrillo, J.~A., Craig, K., and Yao, Y.}
\newblock Aggregation-diffusion equations: {D}ynamics, asymptotics, and singular limits.
\newblock {\em Active Particles, Volume 2: Advances in Theory, Models, and Applications\/} (2019), 65--108.

\bibitem{carrillo2024aggregation}
{\sc Carrillo, J.~A., Fern{\'a}ndez-Jim{\'e}nez, A., and G{\'o}mez-Castro, D.}
\newblock Aggregation-diffusion equations with saturation.
\newblock {\em arXiv preprint arXiv:2410.10040\/} (2024).

\bibitem{carrillo2020long}
{\sc Carrillo, J.~A., Gvalani, R., Pavliotis, G., and Schlichting, A.}
\newblock Long-time behaviour and phase transitions for the {M}ckean--{V}lasov equation on the torus.
\newblock {\em Archive for Rational Mechanics and Analysis 235}, 1 (2020), 635--690.

\bibitem{carrillo2021phase}
{\sc Carrillo, J.~A., and Gvalani, R.~S.}
\newblock Phase transitions for nonlinear nonlocal aggregation-diffusion equations.
\newblock {\em Communications in Mathematical Physics 382}, 1 (2021), 485--545.

\bibitem{carrillo2001entropy}
{\sc Carrillo, J.~A., J{\"u}ngel, A., Markowich, P.~A., Toscani, G., and Unterreiter, A.}
\newblock Entropy dissipation methods for degenerate parabolic problems and generalized {S}obolev inequalities.
\newblock {\em Monatshefte f{\"u}r Mathematik 133\/} (2001), 1--82.

\bibitem{carrillo2010nonlinear}
{\sc Carrillo, J.~A., Lisini, S., Savar{\'e}, G., and Slep{\v{c}}ev, D.}
\newblock Nonlinear mobility continuity equations and generalized displacement convexity.
\newblock {\em Journal of Functional Analysis 258}, 4 (2010), 1273--1309.

\bibitem{carrillo2024structure}
{\sc Carrillo, J.~A., Wang, L., and Wei, C.}
\newblock Structure preserving primal dual methods for gradient flows with nonlinear mobility transport distances.
\newblock {\em SIAM Journal on Numerical Analysis 62}, 1 (2024), 376--399.

\bibitem{chayes2013aggregation}
{\sc Chayes, L., Kim, I., and Yao, Y.}
\newblock An aggregation equation with degenerate diffusion: Qualitative property of solutions.
\newblock {\em SIAM Journal on Mathematical Analysis 45}, 5 (2013), 2995--3018.

\bibitem{chayes2010mckean}
{\sc Chayes, L., and Panferov, V.}
\newblock The {M}ckean--{V}lasov equation in finite volume.
\newblock {\em Journal of Statistical Physics 138}, 1 (2010), 351--380.

\bibitem{de1957sulla}
{\sc De~Giorgi, E.}
\newblock Sulla differenziabilit\`a{} e l'analiticit\`a{} delle estremali degli integrali multipli regolari.
\newblock {\em Memorie della Accademia delle Scienze di Torino. Classe di scienze matematiche fisiche e natural 3\/} (1957), 25--43.

\bibitem{di2019deterministic}
{\sc Di~Francesco, M., Fagioli, S., and Radici, E.}
\newblock Deterministic particle approximation for nonlocal transport equations with nonlinear mobility.
\newblock {\em Journal of Differential Equations 266}, 5 (2019), 2830--2868.

\bibitem{di1979regularity}
{\sc DiBenedetto, E.}
\newblock Regularity results for the porous media equation.
\newblock {\em Annali di Matematica Pura ed Applicata 121\/} (1979), 249--262.

\bibitem{di1986local}
{\sc DiBenedetto, E.}
\newblock On the local behaviour of solutions of degenerate parabolic equations with measurable coefficients.
\newblock {\em Annali della Scuola Normale Superiore di Pisa-Classe di Scienze 13}, 3 (1986), 487--535.

\bibitem{dibenedetto1993degenerate}
{\sc DiBenedetto, E.}
\newblock {\em Degenerate parabolic equations}.
\newblock Universitext. Springer-Verlag, New York, 1993.

\bibitem{dibenedetto2012harnack}
{\sc DiBenedetto, E., Gianazza, U., and Vespri, V.}
\newblock {\em Harnack's inequality for degenerate and singular parabolic equations}.
\newblock Springer, 2012.

\bibitem{dolbeault2009new}
{\sc Dolbeault, J., Nazaret, B., and Savar{\'e}, G.}
\newblock A new class of transport distances between measures.
\newblock {\em Calculus of Variations and Partial Differential Equations 34}, 2 (2009), 193--231.

\bibitem{fagioli2018solutions}
{\sc Fagioli, S., and Radici, E.}
\newblock Solutions to aggregation--diffusion equations with nonlinear mobility constructed via a deterministic particle approximation.
\newblock {\em Mathematical Models and Methods in Applied Sciences 28}, 09 (2018), 1801--1829.

\bibitem{fagioli2022gradient}
{\sc Fagioli, S., and Tse, O.}
\newblock On gradient flow and entropy solutions for nonlocal transport equations with nonlinear mobility.
\newblock {\em Nonlinear Analysis 221\/} (2022), 112904.

\bibitem{giacomin1997phase}
{\sc Giacomin, G., and Lebowitz, J.~L.}
\newblock Phase segregation dynamics in particle systems with long range interactions. {I}. {M}acroscopic limits.
\newblock {\em Journal of Statistical Physics 87}, 1-2 (1997), 37--61.

\bibitem{gomez2025convergence}
{\sc G{\'o}mez-Castro, D.}
\newblock Convergence of a finite-volume scheme for aggregation-diffusion equations with saturation.
\newblock {\em arXiv preprint arXiv:2507.11132\/} (2025).

\bibitem{hillen2001global}
{\sc Hillen, T., and Painter, K.}
\newblock Global existence for a parabolic chemotaxis model with prevention of overcrowding.
\newblock {\em Advances in Applied Mathematics 26}, 4 (2001), 280--301.

\bibitem{jordan1998variational}
{\sc Jordan, R., Kinderlehrer, D., and Otto, F.}
\newblock The variational formulation of the {F}okker-{P}lanck equation.
\newblock {\em SIAM Journal on Mathematical Analysis 29}, 1 (1998), 1--17.

\bibitem{lisini2010class}
{\sc Lisini, S., and Marigonda, A.}
\newblock On a class of modified {W}asserstein distances induced by concave mobility functions defined on bounded intervals.
\newblock {\em Manuscripta Mathematica 133}, 1 (2010), 197--224.

\bibitem{otto2001geometry}
{\sc Otto, F.}
\newblock The geometry of dissipative evolution equations: {T}he porous medium equation.
\newblock {\em Communications in Partial Differential Equations 26}, 1-2 (2001), 101--174.

\bibitem{painter2002volume}
{\sc Painter, K.~J., and Hillen, T.}
\newblock Volume-filling and quorum-sensing in models for chemosensitive movement.
\newblock {\em Canadian Applied Mathematics Quarterly 10}, 4 (2002), 501--543.

\bibitem{porzio1993holder}
{\sc Porzio, M.~M., and Vespri, V.}
\newblock H\"older estimates for local solutions of some doubly nonlinear degenerate parabolic equations.
\newblock {\em Journal of Differential Equations 103}, 1 (1993), 146--178.

\bibitem{urbano2001holder}
{\sc Urbano, J.~M.}
\newblock H{\"o}lder continuity of local weak solutions for parabolic equations exhibiting two degeneracies.
\newblock {\em Advances in Differential Equations 6}, 3 (2001), 327--358.

\bibitem{urbano1930method}
{\sc Urbano, J.~M.}
\newblock {\em The method of intrinsic scaling}, vol.~1930 of {\em Lecture Notes in Mathematics}.
\newblock Springer-Verlag, Berlin, 2008.
\newblock A systematic approach to regularity for degenerate and singular PDEs.

\bibitem{wu2001nonlinear}
{\sc Wu, Z., Zhao, J., Yin, J., and Li, H.}
\newblock {\em Nonlinear diffusion equations}.
\newblock World Scientific, 2001.

\bibitem{Zamponi2017}
{\sc Zamponi, N., and J\"ungel, A.}
\newblock Analysis of degenerate cross-diffusion population models with volume filling.
\newblock {\em Annales de l'Institut Henri Poincar\'e{} C. Analyse Non Lin\'eaire 34}, 1 (2017), 1--29.

\end{thebibliography}

\end{document}